\documentclass[11pt]{article}
\usepackage{mainsty}

\newcommand{\CRE}{\textsf{CRE}}
\newcommand{\BRE}{\textsf{BRE}}
\newcommand{\DIM}{\textsf{DIM}}
\newcommand{\HT}{\textsf{HT}}
\newcommand{\cHT}{\textsf{cHT}}
\newcommand{\Opt}{\textsf{Opt}}
\newcommand{\PATE}{\text{PATE}}

\title{Sharp Minimax Theory for Randomized Experiments}

\author{Timothy Sudijono\footnote{Email: \url{tsudijon@wharton.upenn.edu}}, Edgar Dobriban, Eric Tchetgen Tchetgen}
\date{\today}

\begin{document}
\maketitle
\begin{abstract}
We study minimax-optimal designs and estimators for estimating the sample average treatment effect in finite population randomized experiments, where both design and estimator are unrestricted. For binary potential outcomes, we show this minimax risk is equivalent to the minimax risk $\rho_n^*$ of an estimation problem with two unknown parameters. A similar reduction holds when extending to bounded potential outcomes and to survey experiments.  We leverage this reduction to establish a second-order risk expansion $\rho_n^* = n^{-1} - K_An^{-4/3} + o_n(n^{-4/3})$ for an explicit constant $K_A$ related to the Airy function. The minimax risk is attained by Bernoulli randomization with a nonlinear shrinkage estimator. Our results show that standard procedures such as complete randomization with difference in means are only minimax optimal up to first order in $n.$ We derive further results on admissibility of these procedures and discuss the practical implications of our results.
\end{abstract}


\section{Introduction}
\label{sec:introduction}

The fundamental goal of randomized experiments is to estimate the effect of an intervention on a population of interest \cite{fisher1935design}. Two factors are crucial to this enterprise: the choice of experimental design used to randomize the treatment and control, and the choice of estimator used to analyze the data. Naturally, a vast literature has developed on the optimality of experimental designs, and separately, estimators, under various criteria \cite{elfving1952optimum, kiefer1974general, pukelsheim2006optimal, wu1981robustness, li1983minimaxity, hooper1989minimaxity, waite2022minimax, karwa2023admissibility, bai2022optimality, basse2023minimax, kallus2018optimal, kallus2021optimality}. However, our understanding of optimal \textit{joint} selection over the space of feasible estimators \emph{and} designs remains incomplete. 

An important step towards closing this gap is to characterize the minimax risk of estimating the sample average treatment effect (SATE), optimizing over all choices of estimators and designs. Understanding the minimax risk is important: it serves as a theoretical benchmark of efficiency, allowing us to distinguish between designs and estimators. 
This question has been investigated only recently in \cite{kandiros2026design} for finite-population network experiments and in special cases by \cite{bai2023randomize, kallus2018optimal, kallus2021optimality}. In this paper, we provide a precise characterization of this minimax risk under no interference. We determine the sharp constant and second-order term in the minimax risk and show that an optimal choice of design and estimator is given by balanced Bernoulli randomization with a nonlinear shrinkage estimator.

\subsection{Main Results}
Following the potential outcome framework \cite{neyman1923application, rubin1974estimating}, let $Y_i(a), i=1,\dots,n$ denote a fixed binary potential outcome had unit $i$ received treatment $a=0,1$. Let $\tau=\frac{1}{n} \sum_{i=1}^n \left(Y_i(1) - Y_i(0)\right)$ be the sample average treatment effect, and $A_1,\dots,A_n$ be treatment indicators with joint law $\scr{A}$, which we call the \textit{design}. 
We observe data $Y_i = A_iY_i(1) + (1-A_i)Y_i(0)$ and wish to estimate $\tau$ using an estimator $\hat \tau :\{0,1\}^n \times \{0,1\}^n \mapsto \mathbb{R}$, 
$\hat \tau:(A,Y)\mapsto \hat \tau(A,Y)$. We adopt the terminology in \cite{kandiros2026design} and call a joint choice of design and estimator a \textit{procedure}. For any procedure $(\scr{A},\hat \tau)$, let $R_n(\scr{A},\hat \tau)$ denote its risk $\E[(\hat \tau(A,Y) - \tau)^2]$, where the expectation is taken over the randomness in the design. 
A natural, important question is to characterize the minimax risk
\begin{equation}
\label{eq:minimax_risk_joint_estimator_design}
\frak{R}_n^* := \inf_{\scr{A}, \hat \tau} \sup_{ \cl{P} \in \set{0,1}^{2n}} R_n(\scr{A},\hat \tau).
\end{equation}
Here, $\cl{P} = \set{Y_i(1),Y_i(0)}_{i=1}^n$ denotes a configuration of potential outcomes and the infimum is taken over all choices of procedure. 
The minimax risk is the best possible performance achievable in the worst case over all possible potential outcomes. 

We show that \eqref{eq:minimax_risk_joint_estimator_design} admits a decisive simplification. It is equal to the minimax risk $\rho_n^*$ of the following reduced model: with unknown nonnegative integer parameters $p,q,r$ summing to $n$, we observe data $X = p + \Bin(r,1/2)$ and wish to estimate $(p-q)/n$. For moderate $n$, it is feasible to exactly compute $\rho_n^*$ as well as optimal procedures. A minimax optimal procedure is Bernoulli randomization with equal treatment probability and a nonlinear function applied to the transformed data $\sum_{i=1}^n (A_iY_i + (1-A_i)(1-Y_i)).$ 
We denote this procedure by $(\BRE, \Opt)$.
The nonlinearity corresponds to shrinkage under a least favorable prior. These results are discussed in Section \ref{sec:reduction_to_simplified_model}, which also establishes similar results for bounded potential outcomes, whether discrete or continuous.

On the theoretical front, we establish in Section \ref{sec:second_order_risk} a second-order expansion of the minimax risk 
\begin{equation}
\label{eq:second_order_risk_minimax_intro}
\frak{R}_n^* = \frac{1}{n} - \frac{K_A}{n^{4/3}} + o_n(n^{-4/3}).
\end{equation}
The constant $K_A$ is related to the Airy function,
 a well-known function from mathematical analysis, which is 
one of the linearly independent solutions of the differential equation $d^2y/dx^2 - yx = 0$ \cite{vallee2010airy}.
The asymptotic expansion to second order is evocative of results in the minimax estimation of bounded normal means \cite{bickel1981minimax} and other settings \cite{johnstone1992minimax}. The proof follows a strategy from the work of Levit \cite{levit1981asymptotic,levit1983minimax,levit1986second}, which exposes a connection between the second-order term in a minimax risk expansion and the solution of a problem-specific differential equation, in this case the Airy equation. To our knowledge, the appearance of the $n^{-4/3}$ term and connection to the Airy function is nonstandard and new in this literature.

Our results provide novel perspectives on standard procedures, discussed in Section \ref{sec:comparison_to_standard_methods}. Common procedures in the literature typically use complete randomization (\CRE) or Bernoulli randomization (\BRE). These designs appear in the basic textbooks in the field \cite{imbens2015causal, ding2024first} and have a long, rich tradition \cite{fisher1992arrangement,neyman1923application,kempthorne1955randomization}. For example, Imbens and Rubin name these among the four classical experimental designs \cite{imbens2015causal}. Unless otherwise specified, all usages of the design $\BRE$ will use treatment probability $1/2$, and all instances of $\CRE$ will assume $n$ is even and select $n/2$ units for treatment. These designs are often paired with the difference-in-means (\DIM) or the Horvitz-Thompson (\HT) estimators. Much work in the design-based causal inference literature focuses on these estimators \cite{aronow2013class,aronow2014sharp,lin2013agnostic, athey2017econometrics, imbens2015causal, chattopadhyay2024neymanian, li2017general, harshaw2021algorithmic,AronowChangLopatto2025, bai2024primer}. In light of \eqref{eq:second_order_risk_minimax_intro}, we find that the procedures (\CRE, \DIM), (\BRE, \DIM) are only minimax-optimal to first order. Figure \ref{fig:minimax_risk_curves} displays the maximum risks of these procedures alongside $\rho_n^*$. For small to moderately sized $n$, the maximum risk of these procedures is substantively larger than $\rho_n^*$, which also serves as a lower bound for the maximum risk. For example, at $n = 100$, the maximum risk for (\CRE, \DIM) is about 40\% larger than the optimum. Randomized controlled trials in practice are often of this size \cite{marshall2021state}. 
Complementing our analysis on minimaxity, we show that (\BRE, \Opt) is admissible. On the other hand, (\CRE, \DIM) is also admissible and dominates two other commonly used procedures, (\BRE, \DIM), (\BRE, \HT). Neither (\BRE, \Opt) nor (\CRE, \DIM) dominates the other. 

Finally, in Section \ref{sec:survey_experiments} we consider the setting of survey experiments, where an experimental population is first obtained as a sample from a larger population. Inspired by Theorem \ref{thm:minimax_risk_reduction}, we precisely characterize the minimax optimal choice of survey sampling procedure, experimental design, and estimator of a population treatment effect. As a byproduct, Bernoulli randomization and affine shrinkage estimators are minimax optimal in the superpopulation limit. This provides an interesting contrast with the finite population setting, where nonlinear shrinkage is required.

\begin{figure}
    \centering
    \includegraphics[width=0.8\linewidth]{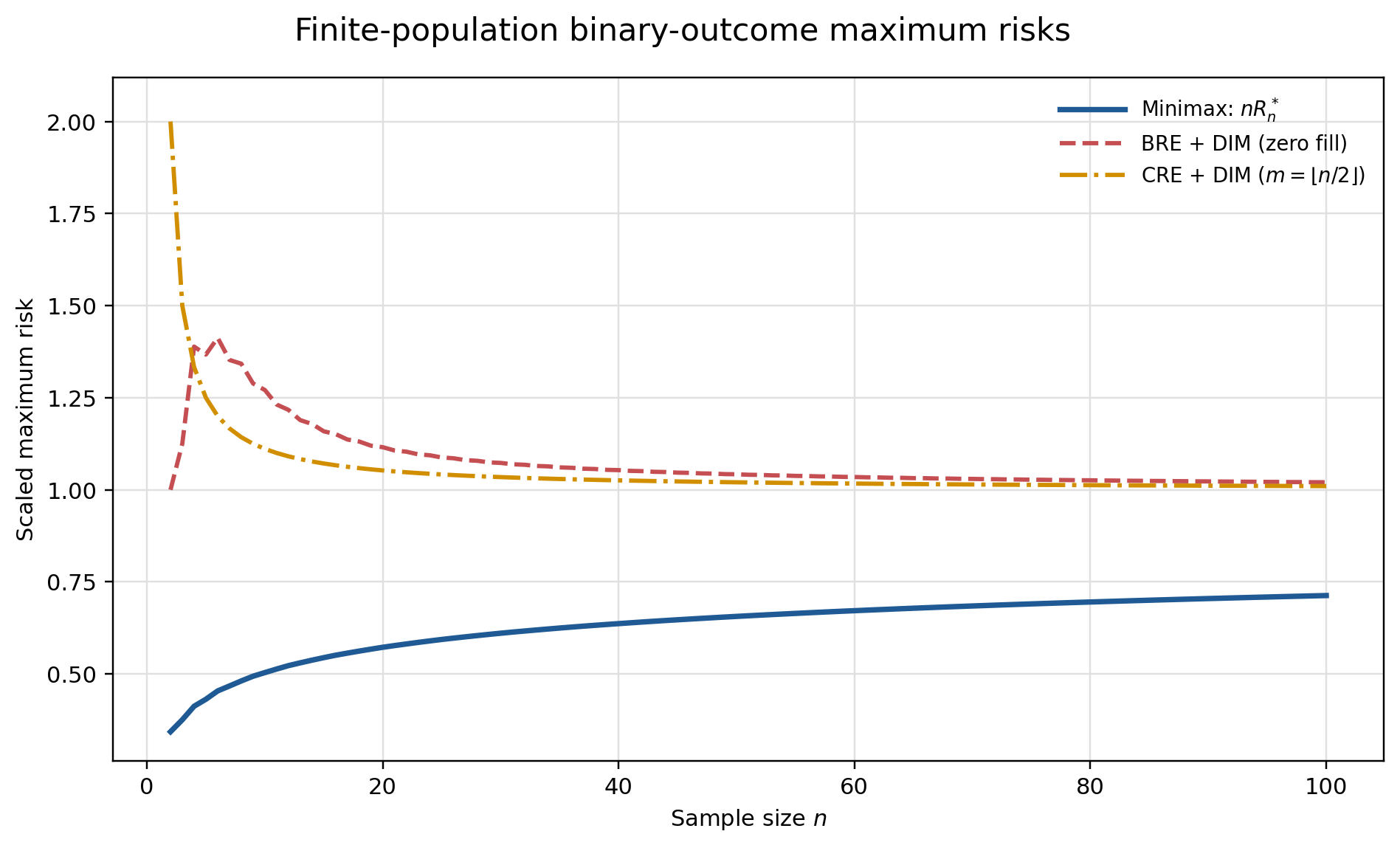}
    \caption{Maximum Risks of Various Designs and Estimators, multiplied by $n$. The maximum risk of \BRE \ with the Horvitz-Thompson estimator is $4/n$ and not plotted here.}
    \label{fig:minimax_risk_curves}
\end{figure}

\subsection{Related Work} 

The most closely related work is a contemporary paper by Hull \cite{hull2026optimal}, who considers the same problem of joint minimax optimal design and estimator in randomized experiments. Hull arrives at an equivalent reduction to Theorems \ref{thm:minimax_risk_reduction} and \ref{thm:minimax_risk_reduction_bounded_case} and obtains the same conclusion that the combination of Bernoulli randomization and a nonlinear estimator is minimax. Notably, the proof strategies of \cite[Theorem 1]{hull2026optimal} and our Theorem \ref{thm:minimax_risk_reduction} are slightly different. Hull utilizes an elegant symmetrization argument, whereas we tightly lower bound the minimax risk by a suitable Bayes risk and emphasize the reduction to the simpler Binomial model. Relative to our work, Hull develops the minimax theory of estimators affine in the observed data $Y$ and provides insights into the optimality of Bernoulli randomization relative to more common designs such as complete randomization or paired randomization. Our paper focuses on other aspects: we characterize the minimax risk to second order, exhibit a sequence of asymptotically least favorable priors, establish admissibility of the procedures discussed in the previous section, and extend these results to the setting of survey experiments. 

Prior to \cite{hull2026optimal}, the most closely related studies are due to Kallus \cite{kallus2018optimal,kallus2021optimality}, Bai \cite{bai2023randomize}, and Kandiros et al. \cite{kandiros2026design}. These works have a slightly richer model of nature than the present paper, allowing for covariates, distributions on potential outcomes, or violations of SUTVA. Thus, these works are complementary to our paper. In \cite{kallus2018optimal,kallus2021optimality}, Kallus solves for the minimax optimal design, fixing the estimator to be difference in means. Complete randomization is shown to be minimax optimal under bounded norm assumptions on the potential outcome matrix. Under restrictions on the relationship between the potential outcomes and covariates, other designs are found to be minimax optimal. \cite{bai2023randomize} generalizes some of the previous results. Under a superpopulation framework, Bai proves a general design minimaxity theorem for estimating arbitrary contrasts of potential outcome means, assuming a class of potential outcome distributions invariant under some group of permutations. As a corollary of this general theorem, $(\CRE,\DIM)$ is found to be the minimax optimal procedure restricted to estimators which are \textit{linear} in the observed data $Y$. 

Kandiros et al. \cite{kandiros2026design} tackle our same question of minimax optimality over joint choices of design and estimator for network interference experiments. They introduce the terminology \textit{procedure} to refer to such a joint choice, which we will follow throughout this paper. \cite{kandiros2026design} characterizes optimal minimax rates for estimating the global average treatment effect and other estimands, with upper and lower bounds that depend on properties of the interference graph. Specializing to the SUTVA case gives a minimax rate of $\Theta(1/n)$. The present paper sharpens the SUTVA analysis, precisely characterizing the constant as well as the second-order term. In the interference literature, \cite{karwa2023admissibility} provides another example of decision-theoretic optimality results on standard causal estimators.

The problem of determining a minimax procedure has also been considered in survey sampling \cite{stenger1979minimax, stenger1988asymptotic, rinott2009some}. Proposition 17 of the survey by Rinott \cite{rinott2009some} provides the answer: among procedures sampling $n$ units from a population of size $N$, the minimax optimal procedure is simple random sampling with an affine shrinkage rule introduced in \cite{hodges1982minimax}. The minimax procedure and risk admit explicit formulas, from which the minimax risk admits the expansion
\begin{equation}
\label{eq:survey_sampling_minimax_risk}
\frac{a_\lambda}{N} - \frac{b_\lambda}{N^{3/2}} + o(N^{-3/2}),
\end{equation}
for a finite population with elements in $[0,1]$ and $n = \floor{\lambda N}, \lambda \in (0,1)$, with some constants $a_\lambda,b_{\lambda}.$ This result serves as an interesting point of comparison to the randomized experiment setting. The second-order term decays faster than in the randomized experiment setting, and  
affine shrinkage is optimal instead of nonlinear shrinkage. The recent paper
\cite{aronow2026minimax} extends these results, deriving the minimax optimal choice of procedure under unbiasedness restrictions and non-symmetric parameter spaces. Therein, they show the centered version of the Horvitz-Thompson estimator \cite{aronow2013class} is minimax optimal. This estimator also plays a role in the randomized experiment setting, appearing in this work as well as \cite{hull2026optimal}:
\begin{equation}
\label{eq:centered_horvitz_thompson}
\hat \tau_{\text{\cHT}}(A,Y) := \frac{\sum_{i=1}^n A_i (Y_i - 1/2)}{n/2} - \frac{\sum_{i=1}^n (1-A_i) (Y_i-1/2)}{n/2}.
\end{equation}
We also note that results separately optimizing the sampling design or estimator abound in this literature \cite{godambe1965admissibility, hodges1982minimax, cheng1983minimax, cheng1987optimality, bickel2011minimax, joshi1979best, scott1975minimax}.

\section{Reduction to a Simplified Model}
\label{sec:reduction_to_simplified_model}

In this section, we describe the reduced two-parameter model with minimax risk $\rho_n^*$ and analyze its properties. The connection to the original procedure problem is made below. Consider first an intermediate statistical model with unknown parameters $\Delta_1,\dots,\Delta_n \in \set{-1,0,1}$ and data $(C_1,\dots,C_n)$ given by 
\begin{equation}
\label{eq:C_definition} 
\begin{cases}
C_i = 1 & \text{ if } \Delta_i = 1 \\
C_i = 0 & \text{ if } \Delta_i = -1 \\
C_i \sim \textsf{Ber}(1/2) & \text{ if } \Delta_i = 0,
\end{cases}
\end{equation}
all mutually independent. 
Here, $\Delta_i$ will be interpreted as the individual treatment effects (ITEs) $Y_i(1) - Y_i(0)$ in the context of the original problem. Let $p$ denote the number of indices $i$ with $\Delta_i = 1,$ $q$ the number of indices with $\Delta_i = -1$ and $r$ the remainder. We are interested in estimating 
\[
\tau(\Delta) := \frac{1}{n}\sum_{i=1}^n \Delta_i = \frac{p - q}{n}.
\]
Let $r_n^*$ denote the minimax risk for estimating $\tau(\Delta)$ given data $(C_1,\dots,C_n)$. By an invariance argument, it suffices to keep track of the parameters $p,q,r$ and to use estimators which are functions of $\sum_{i=1}^n C_i$. Therefore, it suffices to study the minimax risk of the even simpler model with parameters $(p,q,r)$ in the index set
\begin{equation}
\label{eq:index_set}
\cl{I}_n := \set{(p,q,r) \in \bb{Z}_{\geq 0}^3: p+q+r = n},
\end{equation}
where we observe data $X \sim p + \Bin(r,1/2)$, with estimand $(p-q)/n$. Equivalently, $2X - n$ has law $p-q + \sum_{i=1}^r \e_i$, with $\e_i$ independent Rademacher random variables. Let $\rho_n^*$ be the minimax risk of this binomial experiment, which we will refer to as the \textit{reduced model} throughout. Then:

\begin{theorem}[From the intermediate to the reduced model]
\label{thm:simplified_model_minimax_risk}
The minimax risks satisfy
\begin{equation}
\label{eq:simplified_minimax_risk}
    r_n^* = \rho_n^* = \inf_{f:\set{0,\dots,n} \to[-1,1]} \sup_{\cl{I}_n} \E\left(f(p + \textsf{Bin}(r,1/2)) - \frac{p-q}{n} \right)^2.
\end{equation}
Moreover, there is a unique $f^*_n$ minimizing \eqref{eq:simplified_minimax_risk} and it is the posterior mean of $(p-q)/n$ under a least favorable prior $\pi_n^*$ on $\cl{I}_n$. As a consequence, $f^*_n$ is admissible and $\rho_n^*$ is the Bayes risk of  $\pi_n^*.$
\end{theorem}

Although the minimax estimator $f_n^*$ is unique, the least favorable prior need not be unique. 
As we will show, computing the minimax optimal estimator and least favorable prior in the reduced model can be done using convex optimization for moderate values of $n$ below a few hundred. 
In contrast, directly attempting to solve \eqref{eq:minimax_risk_joint_estimator_design} is computationally prohibitive. Figure \ref{fig:numerical_solutions_minimax} shows the numerical solutions $f_n^*$ to \eqref{eq:simplified_minimax_risk}, rewritten as functions of the natural unbiased estimator $(2\sum_{i=1}^n C_i - n)/n$. It is clear from the right-hand panel of the figure that the optimal minimax estimators nonlinearly shrink the unbiased estimator $(2\sum_{i=1}^n C_i - n)/n$ towards zero.

\begin{figure}
    \centering
    \includegraphics[width=\linewidth]{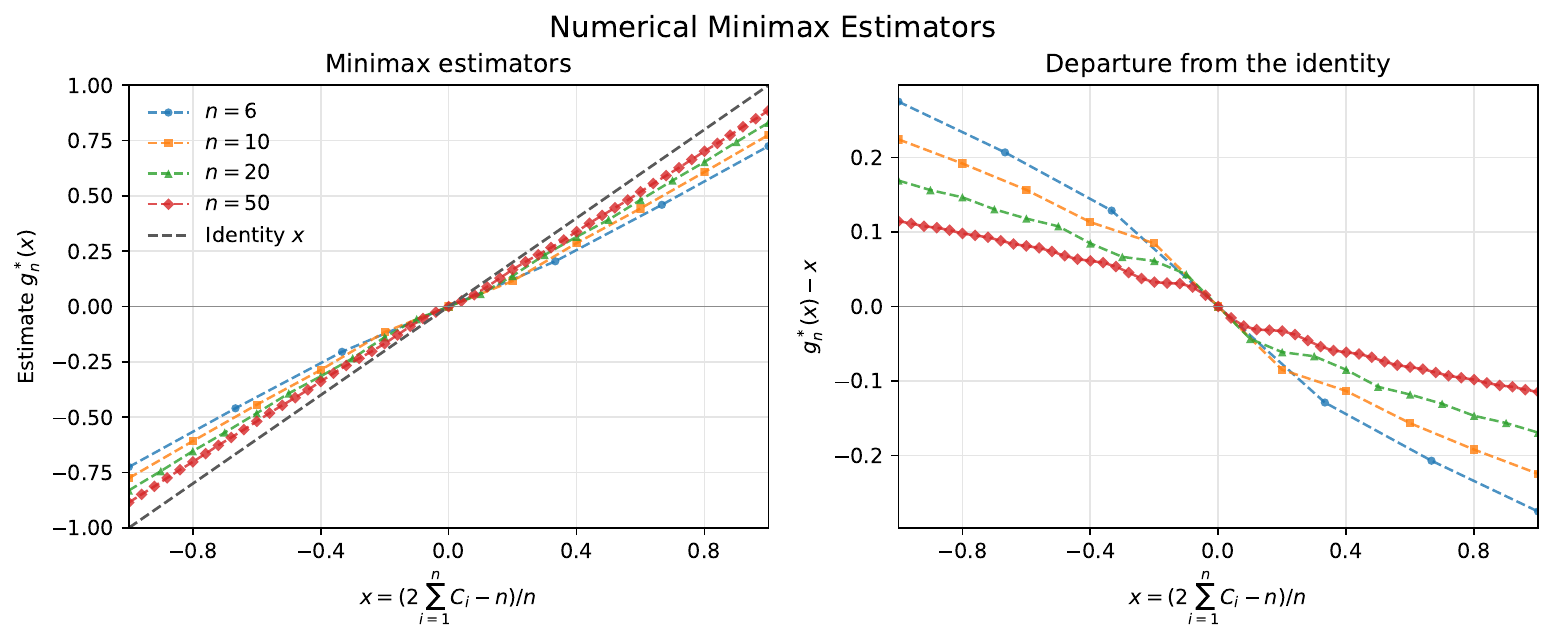}
    \caption{Numerical solutions of the minimax estimator for various values of $n$, displayed as functions of $(2\sum_{i=1}^n C_i - n)/n$. In particular, $g_n^*(x) := f_n^*\left( n(x + 1)/2\right)$ is plotted. The right panel shows deviation from the identity and emphasizes the nonlinear shrinkage.}
    \label{fig:numerical_solutions_minimax}
\end{figure}

The equivalence between the reduced model and randomized experiments is formally made by treating the transformed data $A_iY_i + (1 - A_i)(1 - Y_i)$ as the observations $C_i$ and interpreting $\Delta_i$ as the ITE $Y_i(1) - Y_i(0)$. 


\begin{theorem}[Characterization of the minimax risk]
\label{thm:minimax_risk_reduction}
The minimax risk $\rho_n^*$ in the randomized experiment model from \eqref{eq:minimax_risk_joint_estimator_design} equals that in the binomial experiment described above: 
\begin{equation}
\label{eq:minimax_risk_is_equal_to_simplified_minimax_risk}
\frak{R}_n^* = r_n^* = \rho_n^*.
\end{equation}
Let $f_n^*$ be the function minimizing \eqref{eq:simplified_minimax_risk} in the reduced model.
Then, a minimax optimal procedure is Bernoulli$(1/2)$ randomization with corresponding estimator $\hat{\tau}_{\Opt}$, defined by $\hat{\tau}_{\Opt}(A,Y) :=f^*_n(\sum_{i=1}^n C_i)$. Furthermore, this procedure is admissible.
\end{theorem}

We will refer to this minimax admissible procedure as (\BRE, \Opt). A proof sketch of Eq. \eqref{eq:minimax_risk_is_equal_to_simplified_minimax_risk} is instructive. The idea for the lower bound is to consider a particular prior on the potential outcomes $Y_i(1),Y_i(0)$, which exactly mirrors the reduced model. For any fixed vector $\Delta \in \set{-1,0,1}^n$, generate $(Y_i(1),Y_i(0))$ according to 
\begin{equation}
\label{eq:bayes_model}
(Y_i(1),Y_i(0)) = 
\begin{cases}
(1,0) & \text{ if } \Delta_i = 1 \\
(0,1) & \text{ if } \Delta_i = -1 \\
(0,0) \text{ or } (1,1) \text{ w.p. } 1/2 & \text{ if } \Delta_i = 0.
\end{cases}
\end{equation}
The crucial property is that the transformed quantities $C_i := A_iY_i(1) + (1 - A_i)(1-Y_i(0))$ are independent of the design $A$, under this model. There is a bijection between $(A,Y)$ and $(A,C)$, and so estimators which depend only on $C$ may be considered. The optimal estimator is then the posterior mean of $\tau,$ which equals $\frac{1}{n}\sum_{i=1}^n \Delta_i$, given the data $C$.  This yields the Bayes risk of the reduced model under a particular prior on $\Delta$. Choosing a least favorable prior and applying a finite minimax theorem shows the lower bound. The upper bound is achieved simply by taking the $\BRE(1/2)$ design, under which the quantities $C_i$ have exactly the law in \eqref{eq:C_definition}.

Although the setting of binary outcomes is already central in causal inference \cite{ding2024first, imbens2015causal}, our results can be extended further to bounded potential outcomes. The next result characterizes the minimax risk when the potential outcomes are known to lie in an interval $[L,U].$ Define $\tilde{Y}_i = (Y_i - L)/(U-L)$ and $\tilde{C}_i = A_i \tilde{Y}_i + (1 - A_i)(1 - \tilde{Y}_i)$.

\begin{theorem}[Bounded potential outcomes]
\label{thm:minimax_risk_reduction_bounded_case}
For any finite $U > L$, we have
\begin{equation}
\inf_{\scr{A}, \hat \tau} \sup_{\cl{P} \in [L,U]^{2n}} R_n(\scr{A},\hat \tau) = (U - L)^2r_n^*.
\end{equation}
A minimax optimal procedure is $\BRE(1/2)$ with the estimator $\hat \tau$, defined by
$\hat \tau(A,Y) = (U-L)\E\left[f^*_n\left(\sum_{i=1}^n B_i\right) \mid \tilde C \right]$, with  $B_i \stackrel{ind}{\sim} \Ber(\tilde C_i).$
\end{theorem}
If the outcomes are unbounded, the theorem immediately implies that the worst-case risk is infinite.

\section{Analysis of Minimax Risk}
\label{sec:second_order_risk}

Figure \ref{fig:minimax_risk_curves} shows the curious fact that $n\rho_n^*$ increases rather slowly to its limiting value of $1$. The next theorem, our main result, provides an explanation by expanding the minimax risk up to second order. Its proof provides insight into the form of the optimal minimax estimator and the associated least favorable prior.

\begin{theorem}[Second-order behavior of the minimax risk]
\label{thm:second_order_minimax_risk_behavior}
Let $a_1'$ be the largest negative zero of the derivative of the Airy function. The minimax risk satisfies
\[
\frak{R}_n^* = \frac{1}{n} - \frac{K_A}{n^{4/3}} + o_n(n^{-4/3}),
\]
with $K_A = -4^{1/3}a_1' \approx 1.617$.
\end{theorem}

To give some intuition for Theorem \ref{thm:second_order_minimax_risk_behavior}, we offer the following heuristic. We first reparametrize the model of Section \ref{sec:reduction_to_simplified_model} by considering $2\sum_i C_i - n = \sum_{i=1}^r \e_i + (p-q)$ with $\e_i$ i.i.d. Rademacher $1/2$ random variables. Let $\theta := (p-q)$. Then $\rho_n^*$ is equivalent to the minimax risk of estimating $\theta/n$ on the basis of the observation $X = \left(\theta + \sum_{i=1}^r \e_i\right)$ with unknown parameters $(\theta,r)$ in the parameter space
\begin{equation}
    \Theta_n := \set{(\theta,r): 0 \leq r \leq n,\, |\theta| + r \leq n,\, \theta + r \equiv n \mod 2}.
\end{equation}
For each $(\theta,r) \in \Theta_n$ there exists $p,q,r$ such that $p - q = \theta$ and $p,q,r \geq 0 , p + q + r = n,$ so this is indeed an equivalent parametrization.

In this model, it is not difficult to show that $\rho_n^* = (1 + o(1))/n$ with the unbiased estimator $X/n$ being first-order minimax with maximum risk $1/n$. Therefore, we expect estimators achieving minimax risk up to second order to be small perturbations of $X/n$. Consider the candidate estimator $\delta_{n,h}: [-n,n] \cap \bb{Z} \rightarrow [-1,1]$ defined by
\begin{equation}
\label{eq:second_order_minimax_estimator_h}
\delta_{n,h}(X) := \frac{X}{n} - n^{-\alpha} h(X/n^\alpha),
\end{equation}
for some function $h$ and $\alpha \geq 0$. To heuristically choose $h$ and $\alpha$, suppose $h$ is sufficiently well-behaved so that $h'$ and $h^2$ are Lipschitz. Then the risk of $\delta_{n,h}$ can be expanded by using Lemma \ref{lemma:stein_rademacher_sums}, a version of Stein's identity for i.i.d. sums of Rademacher random variables. Let $U_r = \sum_{i=1}^r \e_i$, $\tilde{\theta} = \theta/n^\alpha$. We have 
\begin{align*}
    \E[(\delta_{n,h}(X) - \theta/n)^2] & = \E\left[\left(\frac{U_r}{n} - n^{-\alpha} h(X/n^{\alpha})\right)^2 \right] \\
    & = \frac{r}{n^2} - 2n^{-\alpha - 1}\E[U_r h(\tilde \theta + U_r/n^\alpha)] + n^{-2\alpha}\E[h(\tilde \theta + U_r/n^\alpha)^2] \\
    & \approx \frac{r}{n^2} - 2rn^{-2\alpha - 1} \E[h'(\tilde{\theta} + U_{r}/n^{\alpha})] + n^{-2\alpha}h(\tilde \theta)^2 + O(n^{-3\alpha}\sqrt{r}) \\
    & = \frac{r}{n^2} - 2rn^{-2\alpha - 1} h'(\tilde \theta)+ n^{-2\alpha}h(\tilde \theta)^2 + O(n^{-3\alpha}\sqrt{r}).
\end{align*}
It is intuitively clear that the hardest parameters $(\theta,r)$ are those on the boundary $r = n - |\theta|.$ Thus, ignoring lower order terms, the worst-case risk is given by
\begin{equation}
\sup_{\tilde \theta} \frac{1}{n} - \frac{|\tilde \theta|}{n^{2-\alpha}} - \frac{2h'(\tilde \theta)}{n^{2\alpha}} + \frac{h(\tilde \theta)^2}{n^{2\alpha}}.
\end{equation}
The sharpest upper bound on the minimax risk is then obtained by solving the following problem:
\[
\frac{1}{n} + \inf_h \sup_{x} \left\lbrace \frac{h(x)^2 - 2h'(x)}{n^{2\alpha}} - \frac{|x|}{n^{2-\alpha}} \right \rbrace.
\]
Balancing the terms forces $\alpha = 2/3$. Indeed, for the latter term to remain second-order, we require $\alpha \in (1/2,1).$ For $\alpha < 2/3$, the second-order term in the minimax problem is $\inf_h \sup_x (h(x)^2 - 2h'(x))/n^{2\alpha}$, because the $1/n^{2\alpha}$ 
term dominates the 
$1/n^{2-\alpha}$ term
above. 
It can be shown that the solution is trivial and given by $h = 0.$ If $\alpha > 2/3$, the second-order term is $-|x|/n^{2-\alpha}$ and the perturbation is irrelevant. Therefore, a second-order risk improvement is obtained only when $\alpha = 2/3$, which yields
\begin{equation}
\frac{1}{n} + \inf_h \sup_{x} \left\lbrace \frac{h(x)^2 - 2h'(x) - |x|}{n^{4/3}} \right \rbrace.
\end{equation}
The question is reduced to the variational problem of choosing $h$ to minimize $M(h) := \sup_{x \in \bR} \{h(x)^2 - 2h'(x) - |x|\}.$ By integrating against smooth test functions $\phi^2$,
 using integration by parts to obtain $\int \phi^2 h' = -2 \int \phi \phi' h$
and completing the square $(\phi h + 2\phi')^2$, 
\begin{align}
\label{eq:variational_problem}
M(h) & \geq -  \inf_{\norm{\phi}_2 = 1} \cl{E}(\phi), \\
\label{eq:quadratic_form}
\cl{E}(\phi) & := \int_\bR 4\phi'(x)^2 + |x|\phi(x)^2 dx.
\end{align}
With $\phi_A$ the optimizer of $\inf_{\phi} \cl{E}(\phi)$, equality in \eqref{eq:variational_problem} is obtained if and only if $h(x) = -2\phi_A'(x)/\phi_A(x)$. The quantity $\cl{E}(\phi)$ is the quadratic form of the second-order differential operator $-4\frac{d^2}{dx^2} + |x|$, i.e., for a sufficiently regular function $\phi$, $\mathcal E(\phi)=\langle \phi,\cl{E}\phi\rangle_{L^2(\mathbb R)}$. 
By the Rayleigh quotient characterization, the value of the variational problem $\inf_\phi \cl{E}(\phi)$ is the smallest eigenvalue of the differential operator, which
is precisely the constant $K_A$ in Theorem \ref{thm:second_order_minimax_risk_behavior}. The corresponding eigenfunction $\phi_A$ satisfies $-4\phi_A''(x) + |x|\phi_A(x) = K_A \phi_A(x).$ 
For $x > 0$, reparametrizing the function as $\tilde{\phi}_A(x) = \phi_A(4^{1/3}x + K_A)$ shows that $\tilde{\phi}_A$ satisfies 
\begin{equation}
\tilde{\phi}_A''(x) - x\tilde{\phi}_A(x) = 0,
\end{equation}
which is the Airy differential equation. A similar reparametrization works when $x < 0.$ Since $\phi_A \in \bL^2$, the solution is thus given by
\begin{equation}
\phi_A(x) \propto \Ai\left( \frac{|x| - K_A}{4^{1/3}}\right).
\end{equation}
Taking $h = -2\phi_A(x)'/\phi_A(x)$ in \eqref{eq:second_order_minimax_estimator_h} yields the candidate second-order minimax estimator. 

To establish a corresponding lower bound in Theorem \ref{thm:second_order_minimax_risk_behavior}, we consider a sequence of priors whose mass functions are discretizations of $\phi_A^2$, which approximate the least favorable distributions. We use a model-specific Van Trees' inequality to lower bound the Bayes risks. 

This tight interplay between second-order differential equations and variational problems appears throughout the literature on second-order minimax risk expansions. Essentially the same strategy was exploited for bounded normal mean estimation, pioneered in a sequence of papers by Levit \cite{levit1981asymptotic, levit1983minimax, levit1986second} and also Bickel \cite{bickel1981minimax}. Koranyi and MacGibbon provide some discussion of this strategy \cite{koranyi2002asymptotically}. In this literature, the Gaussian structure is exploited by using Brown's identity to relate Bayes risk and prior information. In our proof, we instead use Van Trees' inequality, a softer tool. A similar interplay happens in \cite{johnstone1992minimax, Aslan2006} outside of the Gaussian setting. 

To our best knowledge, the appearance of the Airy function and the slower $n^{-4/3}$ second-order term is novel. The primary reason underlying this phenomenon is that the maximal variance $\Var(X/n) \leq (n - |\theta|)/n^2 = \frac{1}{n} - \frac{|\theta|}{n^2}$ at a fixed $\theta$ decays linearly in $|\theta|$. In estimation problems with bounded parameter spaces like the bounded normal means problem, the variance is typically constant.

\begin{remark}[On Affine Estimators]
The results of Hull \cite{hull2026optimal} pay special attention to minimax optimal procedures where the estimator is affine, of the form $c_1(A) + c_2(A)^\top Y$. In \cite[Corollary 1]{hull2026optimal} it is shown that the optimal procedure is $\BRE(1/2)$ with the affine estimator $\hat\tau_{\text{affine}} := \frac{\sqrt{n}}{\sqrt{n} + 1} \hat \tau_{\cHT}$, where $\hat\tau_{\cHT}$ is defined in  \eqref{eq:centered_horvitz_thompson}. The minimax risk of this optimal affine procedure is $\frac{1}{(1 + \sqrt{n})^2}$. Expanding to second order, we obtain
\begin{equation}
\frac{1}{(1 + \sqrt{n})^2}  = \frac{1}{n} - \frac{2}{n^{3/2}} + O(n^{-2}).
\end{equation}
Like the Hodges-Lehmann survey sampling result in Eq. \eqref{eq:survey_sampling_minimax_risk}, the second-order term here differs from that in Theorem \ref{thm:second_order_minimax_risk_behavior} and shows a qualitative advantage of nonlinear estimators over affine estimators.
\end{remark}

\section{Comparison to Standard Procedures}
\label{sec:comparison_to_standard_methods}

In practice, the most widely used designs are the completely randomized experiment and Bernoulli randomized experiment \cite{ding2024first, imbens2015causal}. Recall that the \CRE \ selects $n_1$ of the $n$ units uniformly at random to be treated, and Bernoulli randomization (\BRE) assigns treatment via indicators $A_i$ drawn i.i.d. $\Ber(p)$.  We will focus on the \CRE \ where $n/2$ units are treated, and on \BRE s with treatment probability $1/2$. For simplicity throughout this section, we will assume $n$ is even. Two commonly used estimators are the difference-in-means $\hat \tau_{\DIM}$ and Horvitz-Thompson estimators $\hat \tau_{\HT}$ defined by
\begin{align}
\label{eq:dim}
 \hat \tau_{\DIM}(A,Y) & := \frac{\sum_{i=1}^n A_i Y_i}{\sum_{i=1}^n A_i} - \frac{\sum_{i=1}^n (1-A_i) Y_i}{\sum_{i=1}^n (1-A_i)}, \\
 \label{eq:horvitz_thompson}
 \hat \tau_{\HT}(A,Y) & := \frac{\sum_{i=1}^n A_i Y_i}{n/2} - \frac{\sum_{i=1}^n (1-A_i) Y_i}{n/2}.
\end{align}

In addition, we will consider the centered Horvitz-Thompson estimator $\hat \tau_{\cHT}$ \eqref{eq:centered_horvitz_thompson}, which has been studied in \cite{aronow2013class, aronow2026minimax}. The estimator $\hat \tau_{\text{\cHT}}$ is a special case of the family of estimators introduced in \cite{aronow2013class} and can be seen as an AIPW-like estimator with outcome models $1/2$ for both treatment and control. It appears in a similar form in the survey sampling literature \cite{aronow2026minimax}, where the corresponding estimator is shown to be minimax optimal over design-unbiased estimators under independent sampling. The estimator arises naturally in our problem: a quick derivation shows that $\hat \tau_{\cHT}(A,Y) = \frac{1}{n}\left(2\sum_{i=1}^n C_i  - n\right)$, which is exactly the natural unbiased estimator in the reduced model of Section \ref{sec:reduction_to_simplified_model}. 

For the \CRE, $\cHT$ is equal to both $\DIM$ and $\HT$. In \BRE, they are distinct. For (\BRE, \DIM), we take the convention $\hat \tau_{\DIM} = 0$ if $\sum_i A_i \in \set{0,n}.$ Note that in the Bernoulli randomized experiment, \HT \ is known to be asymptotically less efficient than \DIM, but we retain it nonetheless. We will therefore consider four procedures: (\CRE, \DIM), (\BRE, \DIM), (\BRE, \cHT) and (\BRE, \HT). By computing their maximum risks, we can show that the three procedures $(\CRE,\DIM), (\BRE,\DIM), (\BRE,\cHT)$ are only minimax to first order, while $(\BRE,\HT)$ is not. 

\begin{proposition}[Maximum Risk]
\label{prop:maximum_risks}
Let $n \geq 4$ be even. Then the following are true:
\begin{align*}
\sup_{\cl{P} \in \set{0,1}^{2n}} R_n(\CRE,\DIM) & = \frac{1}{n-1}, \qquad 
&\sup_{\cl{P} \in \set{0,1}^{2n}} R_n(\BRE,\HT)  = \frac{4}{n}, \\ 
\sup_{\cl{P} \in \set{0,1}^{2n}} R_n(\BRE,\DIM) & = \frac{1}{n} + \frac{2}{n^2} + O(n^{-3}), \qquad
&\sup_{\cl{P} \in \set{0,1}^{2n}} R_n(\BRE,\cHT) = \frac{1}{n}.
\end{align*}
\end{proposition}
The proof of Proposition \ref{prop:maximum_risks} also exposes the worst-case potential outcome configurations for each procedure. For the procedures using $\DIM$ with $n$ even, the worst-case is $n/2$ units with $(Y_i(1),Y_i(0)) = (1,1)$ and $n/2$ units with $(0,0)$. These worst-case configurations have zero SATE and are unique up to relabeling of units. For $(\BRE,\HT)$, the worst-case is  $Y_i(1) = Y_i(0) = 1$ for all $i;$ for $(\BRE,\cHT),$ every configuration for which $Y_i(1) - Y_i(0) = 0, \forall i$ achieves the maximum risk. For $(\BRE,\Opt)$, there are multiple worst-case configurations. Recall that the minimax optimal procedure is a function of $\sum_{i=1}^n C_i$, and under \BRE, $\sum_{i=1}^n C_i \sim p + \Bin(r,1/2)$, with $C_i = Y_iA_i + (1-Y_i)(1-A_i)$ and $p,q,r$ being the number of units with ITE equal to $1,-1,0$ respectively. Then any potential outcome configuration with the same $(p,q,r)$ values shares the same risk. Figure \ref{fig:opt_worst_case_configurations} displays some numerically-computed worst-case values of $(p,q,r)$ for various values of $n$. Note there are worst-case configurations with SATE equal to $-1,0,1$ and various values in between. All configurations appear to satisfy $p = 0$ or $q = 0$ and the set is symmetric upon swapping $p,q$.\\

\begin{figure}
    \centering
    \includegraphics[width=\linewidth]{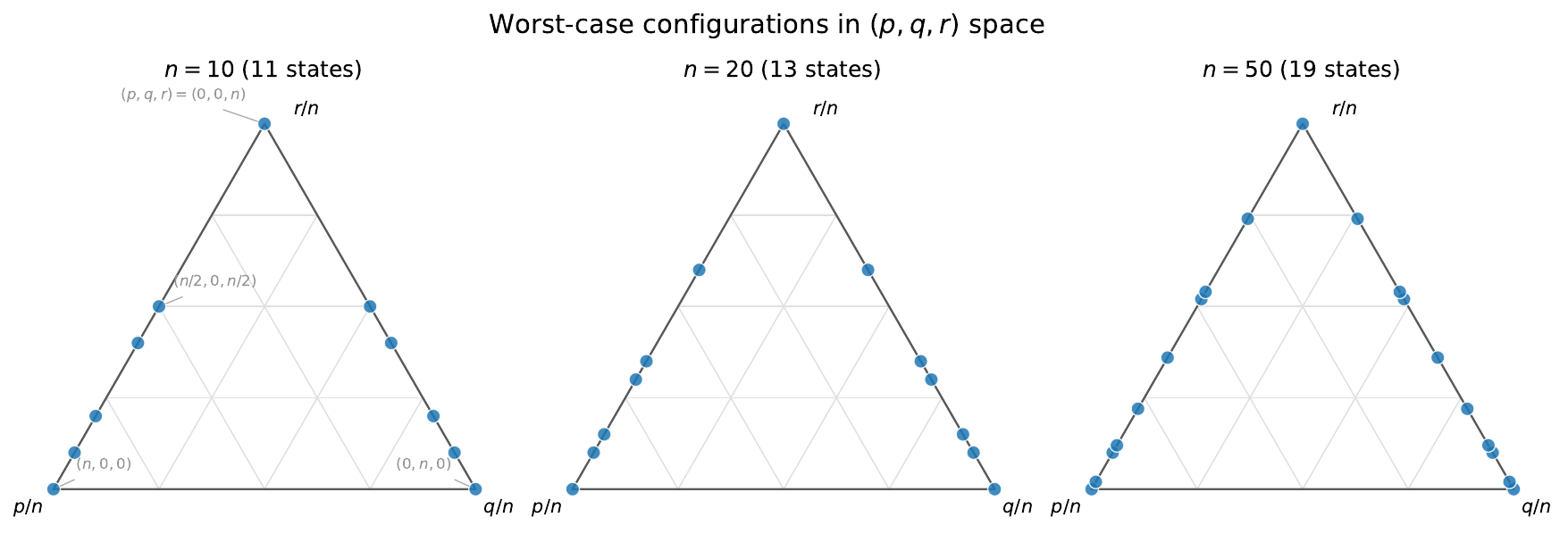}
    \caption{Blue dots represent numerically computed worst-case configurations for $(\BRE,\Opt)$, up to a tolerance of $10^{-7}.$ The plots represent the two-dimensional simplex $p+q+r = n$. Some explicit labels are provided on the left plot.}
    \label{fig:opt_worst_case_configurations}
\end{figure}

Despite its aesthetic and theoretical appeal, minimaxity may be too conservative a criterion to have practical implications -- a practitioner may suspect these worst-case configurations are too unrealistic. A useful complement is to consider admissibility and dominance as another lens. Recall that a procedure $(\scr{A},\hat\tau)$ \textit{dominates} another procedure $(\scr{B},\hat\theta)$ if 
\[
\E_{\scr{A}}\left[\left(\hat\tau(A,Y) - \tau \right)^2\right] \leq \E_{\scr{B}}\left[\left(\hat\theta(A,Y) - \tau \right)^2\right]
\]
for all $\cl{P} \in \set{0,1}^{2n}$,
with strict inequality at at least one configuration $\cl{P} \in \set{0,1}^{2n}.$ A procedure is \textit{admissible} if no other procedure dominates it.

\begin{theorem}[Admissibility of procedures]
\label{thm:admissibility_of_procedures}
For all even $n \geq 4$, (\CRE, \DIM) and (\BRE, \cHT) are admissible. Furthermore (\CRE, \DIM) dominates (\BRE, \DIM), (\BRE, \HT). Therefore, they are inadmissible.
\end{theorem}

The domination of (\BRE, \DIM), (\BRE, \HT), all else equal, suggests we should prefer (\BRE, \Opt), (\BRE, \cHT), or (\CRE, \DIM). Although (\BRE, \Opt) demonstrates a substantial improvement in maximum risk per Figure \ref{fig:minimax_risk_curves}, a clear practical recommendation of which admissible procedure to use requires further research. We provide some numerical comparisons of the risk of these 3 procedures across various subsets of the potential outcome space $\set{0,1}^{2n}$.

The first simulation determines the fraction of the configuration space $\set{0,1}^{2n}$ for which each procedure has the smallest risk amongst all competitors. Two measures are used: the standard uniform measure on the hypercube, and the measure under which count vectors $(n_{11},n_{10},n_{01},n_{00})$ are weighted equally, where $n_{ab} = \#\set{i: Y_i(1)=a,Y_i(0) = b}$. Figure \ref{fig:admissible_procedure_comparison_full_space} shows 
that under the
uniform prior as $n$ increases, (\BRE, \Opt) has smaller risk than the other procedures on nearly every configuration. Under the count-vector weighting scheme, (\CRE, \DIM) appears to be the best performer, although (\BRE, \Opt) is still useful. One reason explaining the phenomenon in the left panel of Figure \ref{fig:admissible_procedure_comparison_full_space} is simply that (\BRE, \Opt) uses a shrinkage estimator and performs well when the SATE is small. Figure \ref{fig:fixed_tau_admissible_procedure_comparison} corroborates this, showing the same comparisons on the subsets of the configuration space with fixed $\tau.$ As $|\tau|$ increases, (\BRE, \Opt) is less competitive than the other procedures. Note that these plots do not show the difference in risk between the three methods. 

\begin{figure}
    \centering
    \includegraphics[width=\linewidth]{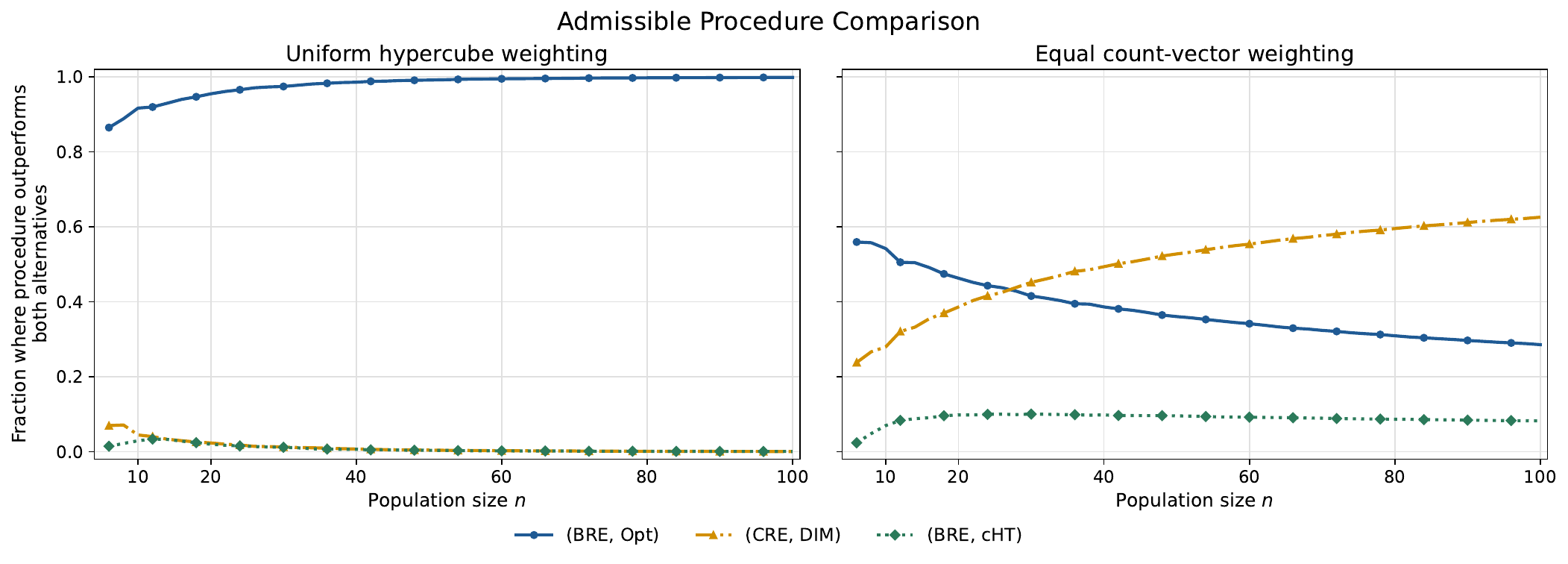}
    \caption{Each line shows the measure of configuration space $\set{0,1}^{2n}$ on which a fixed procedure has strictly smaller risk than its competitors, among the three admissible procedures for even $n$ studied in this paper. Ties are ignored. Two measures are used as in the text.}
    \label{fig:admissible_procedure_comparison_full_space}
\end{figure}

\begin{figure}
    \centering
    \includegraphics[width=\linewidth]{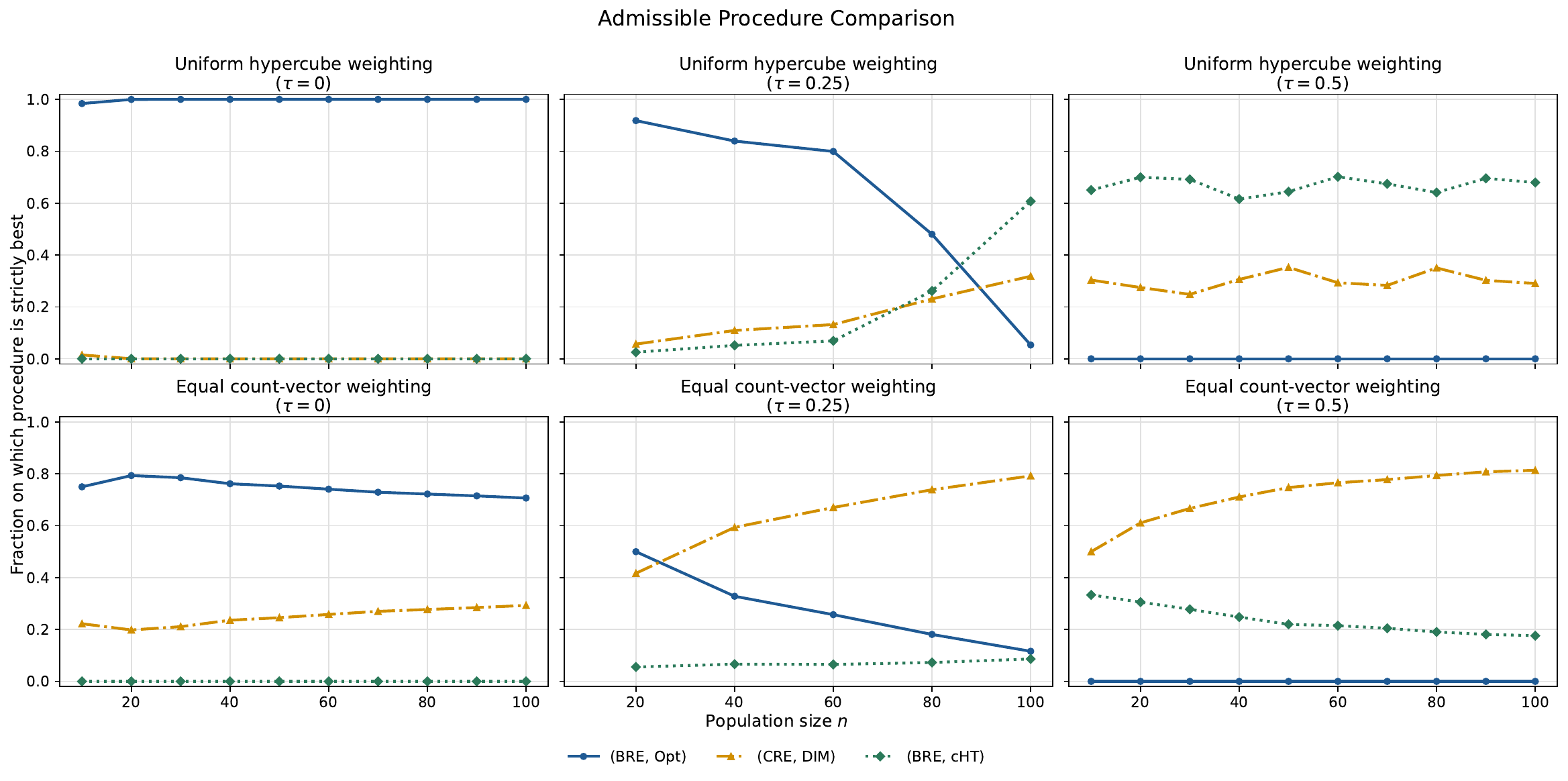}
    \caption{Same as Figure \ref{fig:admissible_procedure_comparison_full_space}, with configuration space restricted to those with fixed $\tau$. Only values of $n$ for which $n\tau \in \bb{Z}$ and which are multiples of $10$ are plotted.}
    \label{fig:fixed_tau_admissible_procedure_comparison}
\end{figure}

The third simulation investigates the performance of (\BRE,\Opt) for inference. We use the Rigdon-Hudgens procedure for binary potential outcomes \cite{rigdon2015randomization}, which inverts level-$\alpha$ randomization tests for sharp nulls $H_0:\Delta = \Delta_0$. Let $p(\Delta)$ be the two-sided randomization $p$-value obtained for testing this sharp null. A confidence interval for $\tau$ is obtained by taking the convex hull of all values $\frac{1}{n}\sum_i \Delta_i$ for $\Delta \in \set{-1,0,1}^n$ such that $p(\Delta) \geq \alpha$. Confidence intervals for (\CRE,\DIM) and (\BRE, \cHT) are obtained in the same way, by resampling the experimental design and evaluating the estimator. This method guarantees finite-sample validity.

Leveraging computational advances for these intervals \cite{rigdon2015randomization,AronowChangLopatto2025,li2016exact}, only $O(n \log n)$ many randomization tests need to be computed for (\CRE,\DIM) and at most $O(n^4)$ for the others. It is then possible to calculate the worst-case expected lengths of each of the resulting confidence intervals, for moderate $n$. Figure \ref{fig:worst_case_confidence_interval_length} shows the results, offering yet another advantage of the minimax optimal procedure: the confidence interval for (\BRE,\Opt) has better worst-case expected length than the other procedures. 

\begin{figure}
    \centering
    \includegraphics[width=0.8\linewidth]{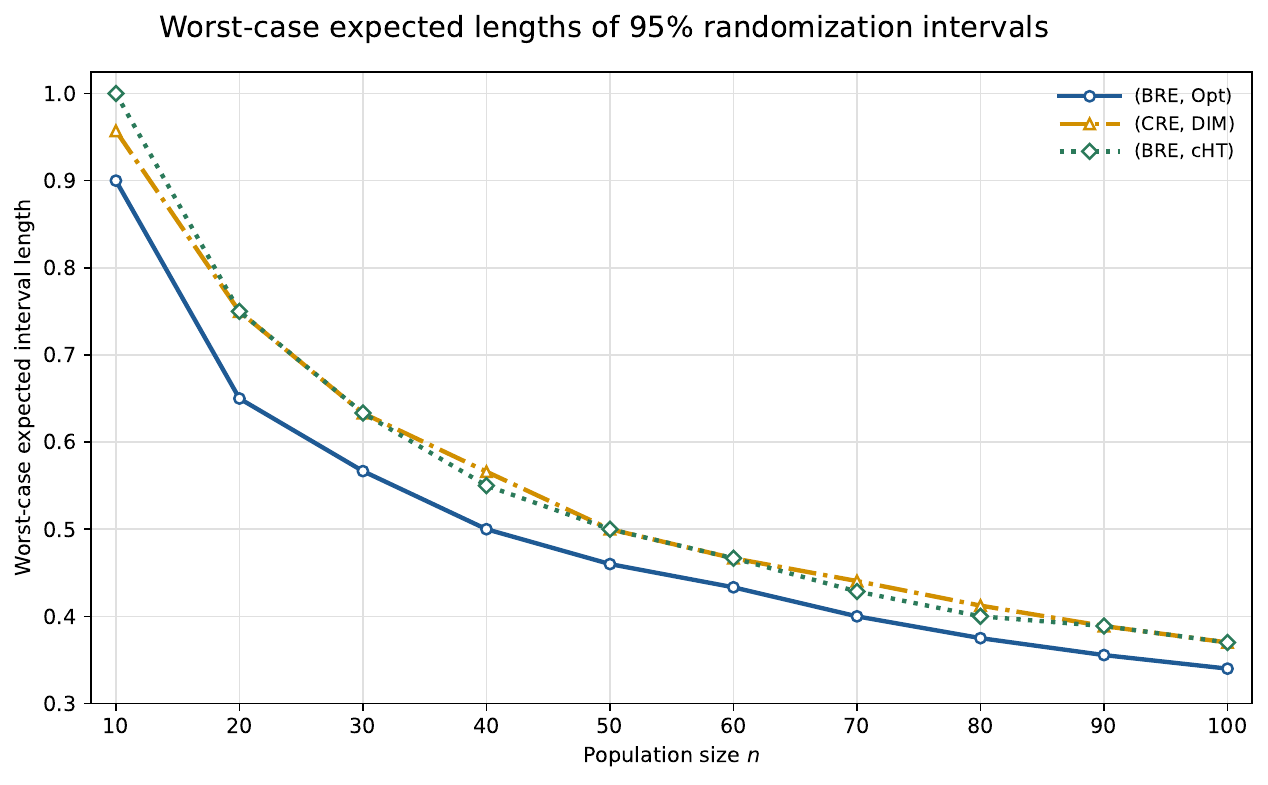}
    \caption{Worst-case expected lengths as a function of $n$ for Rigdon-Hudgens confidence intervals based on inverting randomization tests of sharp nulls. The worst-case is taken over all configurations in $\set{0,1}^{2n}$.}
    \label{fig:worst_case_confidence_interval_length}
\end{figure}

\section{Minimaxity in Survey Experiments}
\label{sec:survey_experiments}

In this section, we extend the results of Section \ref{sec:reduction_to_simplified_model} to the setting of survey experiments, where the experimental units are \textit{sampled} from a larger population of interest. Yang et al. \cite{yang2023rejective}, alongside previous works \cite{imai2008misunderstandings,stuart2018generalizability}, refer to this framework as the ``gold standard for estimating treatment effects from a target population". Survey experiments allow investigators to assess external validity, the extent to which the experimental results generalize beyond the units in the study. They also interpolate between our previous finite population analysis and a superpopulation analysis. See \cite{abadie2020sampling,aronow2014sharp,imbens2021causal,yang2023rejective,cytrynbaum2021optimal,imai2008misunderstandings, stuart2018generalizability} for further background and results on these experiments. 

In a survey experiment, there is a fixed population of $N$ units with unknown binary potential outcomes $\cl{P} = \set{Y_i(1),Y_i(0)}_{i=1}^N.$ Let $\tau_{\text{PATE}}$ denote this population average treatment effect $\frac{1}{N}\sum_{i=1}^N (Y_i(1) - Y_i(0))$. 
Due to budget constraints, only a sample of $n$ units can be recruited for an experiment. The sample is collected via a sampling scheme $\scr{S}$. Concretely, a sampling scheme is a probability distribution on all size $n$ subsets of $[N]$. Let $(Z_i)_{i=1}^N$ be indicators of whether unit $i$ was sampled or not and $S := \set{i:Z_i = 1}$ denote the sample. Once the sample is obtained, we select an experimental design $\scr{A}_S$ that assigns units in $S$ to treatment or control, yielding data $Y_i = A_iY_i(1) + (1-A_i)Y_i(0)$ for $i \in S$. The design may depend on the sample $S$. Finally, an estimator $\hat \tau: (A,Y) \mapsto \hat \tau(A,Y)$ is chosen to estimate $\tau_{\text{PATE}}$.

We can ask for the minimax optimal choice of sampling scheme $\scr{S}$, experimental design $\scr{A}$, and estimator $\hat \tau$ as well as the minimax risk
\begin{equation}
\mathfrak{R}_{N,n} := \inf_{\scr{A},\scr{S},\hat \tau} \sup_{\cl{P}} \E[ (\hat \tau - \tau_{\text{PATE}})^2].
\end{equation}
When $n = N$, we take a census of the entire population, so we recover the finite population question answered in Section \ref{sec:reduction_to_simplified_model}. With $n$ fixed and $N \rightarrow \infty$, the survey experiment approximates a superpopulation setting where the units are roughly drawn i.i.d. from an unknown distribution. 

Remarkably, it is possible to characterize the minimax procedure in this three-step pipeline. With parameters   $(p,q,r) \in \cl{I}_N$ as in Eq. \eqref{eq:index_set}, let $\rho_{N,n}^*$ be the minimax risk of estimating $(p-q)/N$ on the basis of an observation $X$ with law
\begin{equation}
\label{eq:survey_experiment_reduced_model}
X \mid K \sim \textsf{Hypergeom}(N,K,n), \qquad K \sim p + \Bin(r,1/2).
\end{equation}
Let $g^*$ be a minimax optimal estimator in this problem.
\begin{theorem}
\label{thm:survey_experiment_extension}
For every $N \geq n > 0, \mathfrak{R}_{N,n} = \rho_{N,n}^*$. A minimax optimal procedure is given by simple random sampling without replacement of $n$ units, a $\BRE(1/2)$ design on the sampled units, and the estimator
\(
\hat\theta: (A,Z,Y) \mapsto
g^*\left(\sum_{i: Z_i = 1} A_iY_i + (1-A_i)(1-Y_i) \right).
\)
Furthermore, for fixed $n$, $\lim_{N \rightarrow \infty} \mathfrak{R}_{N,n} = \frac{1}{(1+\sqrt{n})^2}.$ 
\end{theorem}  
In the limit $N\rightarrow \infty$, the minimax risk $\mathfrak{R}_{N,n}$ is thus equal to the affine risk, which is the maximum risk of $(\BRE,\hat\tau_{\text{affine}})$. This result provides another motivation for 
Hull's study of $(\BRE,\hat\tau_{\text{affine}})$. 

\begin{remark}
There is a helpful resemblance between the minimax affine theory and the minimax estimation of a binomial parameter \cite[Section 5]{hodges1950some}. With an unknown proportion $p$ and fixed $n$, the minimax risk of estimating $p$ from the observation $X \sim \Bin(n,p)$ is known to be $\frac{1}{4(1+\sqrt{n})^2}$, achieved by the Bayes estimator under a $\Beta(\sqrt{n}/2,\sqrt{n}/2)$ prior for $p$: $\frac{\sqrt{n}}{\sqrt{n} + 1}\frac{X}{n} + \frac{1}{2(\sqrt{n} + 1)}$. We will use the same Beta prior to show the lower bound in $\lim_{N \rightarrow \infty} \mathfrak{R}_{N,n} = \frac{1}{(1+\sqrt{n})^2}.$ 
\end{remark}

\section{Discussion}
\label{sec:discussion}

In this paper, we offer a sharp description of the minimax risk of estimating the sample average treatment effect with bounded potential outcomes, minimizing over joint choices of design and estimator. We show the minimax risk is equal to, up to constant factors, the minimax risk of a simplified experiment in which we observe data $X \sim p + \Bin(r,1/2)$, and wish to estimate $(p - q)/n$, for nonnegative integer parameters $p,q,r$ summing to $n.$ We derive a second-order expansion for the minimax risk of this simplified experiment, revealing the slow second-order rate $n^{-4/3}$ and a deep connection to the Airy function. 

Although our results are intended to be a theoretical benchmark, the procedure $(\BRE,\Opt)$ may potentially be developed into a method for the practical analysis of small randomized controlled trials. The sample size $n$ should be in some intermediate range, on the order of a few hundred, where computing the exact nonlinear rule is not prohibitive and where the rule offers a substantial reduction in maximum risk over $(\CRE,\DIM), (\BRE,\cHT)$. Broad surveys of medical RCTs suggest that the number of participants recruited frequently falls within this range \cite{totton2023review, chang2022survey, marshall2021state}. Our numerical results in Figures \ref{fig:admissible_procedure_comparison_full_space}, \ref{fig:fixed_tau_admissible_procedure_comparison} also suggest $(\BRE,\Opt)$ is a useful procedure for small effect sizes, which are important in practice. 

We envision that the reduction to the binomial experiment in Theorem \ref{thm:minimax_risk_reduction} is useful for related problems. For example, one may be interested in minimax-length confidence intervals which cover $\tau$ and are finite-sample valid, or in treatment-choice problems with linear loss, where a decision must be made to deploy the treatment onto the same population. Extending our results to the setting with covariates would also be interesting, as many experimental designs in practice are motivated by the desire to balance covariates. 

We put forward (\BRE, \Opt) as a provably minimax optimal procedure deserving of more research to determine its practical applicability. In the meantime, our results justify the continued use of (\CRE, \DIM) and (\BRE, \cHT), both of which are first-order minimax optimal, admissible over all procedures, unbiased, and have a well-developed inferential theory. 

\section*{Acknowledgements}

TS thanks Will Bekerman, Lihua Lei, Harrison Li, Bryan Park, Dylan Small, and Maggie Wang for helpful comments. This work was supported in part by the US NSF, ARO, ONR, and the Sloan Foundation. AI assistance was used in the preparation of this work, including generating code and figures, suggesting and checking proofs, and revising the article. The authors wrote the exposition and prose. All proofs were verified by the authors, and we take responsibility for any errors.

\bibliography{main}

\appendix

\section{Computation of the Minimax Estimator}
\label{sec:computation}

Solving the reduced minimax problem in Theorem~ \ref{thm:simplified_model_minimax_risk} is substantially more
tractable than solving the original minimax problem directly. For
$s=(p,q,r)\in\mathcal I_n$, define $\theta_s := \frac{p-q}{n}$
and let
\[
    P_s(k)
    :=
    2^{-r}\binom{r}{k-p}
    \mathbf 1\{p\leq k\leq p+r\},
    \qquad k=0,\ldots,n,
\]
denote the probability of observing $X=k$ under state $s$. Then, for an
estimator $f:\{0,\ldots,n\}\to[-1,1]$, its risk at $s$ is
\[
    R_s(f)
    :=
    \sum_{k=0}^n
    P_s(k)\bigl(f(k)-\theta_s\bigr)^2.
\]

Introducing an epigraph variable $t$, the minimax risk can be written as the following quadratically-constrained convex program:
\begin{equation}
\label{eq:appendix-primal-program}
\begin{aligned}
    \text{minimize}_{f,t}\quad & t \\
    \text{subject to}\quad
    & \sum_{k=0}^n
      P_s(k)\bigl(f(k)-\theta_s\bigr)^2
      \leq t,
      && s\in\mathcal I_n,\\
    & -1\leq f(k)\leq 1,
      && k=0,\ldots,n.
\end{aligned}
\end{equation}
The restriction $f(k)\in[-1,1]$ is without loss of generality, since the
estimand $\theta_s$ lies in $[-1,1]$ and truncation to this interval cannot
increase mean squared error. Each constraint in
\eqref{eq:appendix-primal-program} is convex in $f$. This can be used as a simple way of computing the minimax risk and the minimax optimal estimator $f^*_n.$

Another approach is to solve the dual formulation, which produces both a least favorable prior and the associated Bayes estimator. Let
\(
    \Pi_n
    :=
    \left\{
        \pi\in\mathbb R_+^{\mathcal I_n}:
        \sum_{s\in\mathcal I_n}\pi_s=1
    \right\}
\)
denote the simplex of priors on $\mathcal I_n$. For $\pi\in\Pi_n$, define $a_k(\pi) := \sum_{s \in \cl{I}_n}\pi_sP_s(k), b_k(\pi) := \sum_{s \in \cl{I}_n}\pi_s\theta_sP_s(k)$, and $c(\pi) = \sum_{s \in \cl{I}_n} \pi_s \theta_s^2$. In particular, $a_k(\pi) = \Prob_{\pi}(X = k)$ and whenever $a_k(\pi)>0$,
\[
\frac{b_k(\pi)}{a_k(\pi)} = \mathbb E_\pi[\theta_s\mid X=k].
\]
Note that $a_k(\pi) = 0 \Rightarrow b_k(\pi) = 0.$ The Bayes risk can easily be calculated as
\begin{equation}
\label{eq:appendix-bayes-risk}
 B_n(\pi) = c(\pi) - \sum_{\substack{0\leq k \leq n \\ a_k(\pi)>0} } \frac{ b_k(\pi)^2 }{ a_k(\pi) }.
\end{equation}
Then $B_n(\pi)$ is concave in $\pi$, since $a_k(\pi),b_k(\pi),c(\pi)$ are linear in $\pi$ and the function $\psi$ defined by $\psi(a,b) =  b^2/a, a > 0$ and $\psi(a,b) = 0, (a,b) = (0,0)$ is convex, being the perspective of the convex function $b\mapsto b^2$. By the finite minimax theorem, $\rho_n^* = \max_{\pi\in\Pi_n} B_n(\pi).$ Consequently, a least favorable prior can be computed by solving
\begin{equation}
\label{eq:appendix-full-dual}
\begin{aligned}
    \text{maximize}_{\pi}\quad
    &
    \sum_{s\in\mathcal I_n}\pi_s\theta_s^2 - \sum_{\substack{0\leq k\leq n\\a_k(\pi)>0}}
    \frac{\left(\sum_{s\in\mathcal I_n} \pi_s\theta_sP_s(k) \right)^2}{\sum_{s\in\mathcal I_n}\pi_sP_s(k)}\\
    \text{subject to}\quad
    &
    \pi_s\geq 0,
    \qquad s\in\mathcal I_n,\qquad
    \sum_{s\in\mathcal I_n}\pi_s=1.
\end{aligned}
\end{equation}
This is a standard convex optimization problem. If $\pi_n^*$ is a solution, then the
corresponding minimax optimal estimator $f_n^*:\set{0,\dots,n} \rightarrow [-1,1]$ is defined by  $f_n^*(k) = \frac{b_k(\pi_n^*)}{a_k(\pi_n^*)}.$

\section{Proofs for Section \ref{sec:reduction_to_simplified_model}}

\begin{proof}[Proof of Theorem \ref{thm:simplified_model_minimax_risk}]
First, we will show $r_n^* = \rho_n^*$ by showing it suffices to restrict to permutation invariant estimators of $C$. Let $\delta$ be any estimator as a function of $C$. Consider
\(
\overline{\delta}(C) := \frac{1}{n!} \sum_{\sigma \in S_n} \delta(\sigma C).
\)
Let $P_\Delta$ be the law of $C$ given $\Delta$. Because the law of $C_i \mid \Delta_i$ is the same for each coordinate, $\sigma C \sim P_{\sigma \Delta}.$ 
By Jensen's inequality and the invariance $\tau(\Delta) = \tau(\sigma \Delta)$, we have 
\begin{align*}
\E_\Delta(\overline{\delta}(C) - \tau(\Delta))^2 & \leq \frac{1}{n!} \sum_{\sigma \in S_n} \E_\Delta(\delta(\sigma C) - \tau(\Delta))^2  \\ 
& = \frac{1}{n!} \sum_{\sigma \in S_n} \E_{\sigma\Delta}(\delta(C) - \tau(\sigma\Delta))^2 
 \leq \sup_\Delta \E_\Delta(\delta(C) - \tau(\Delta))^2.
\end{align*}
Since $C \in \set{0,1}^n$, the permutation invariant functions are functions of $\sum C_i$. Therefore, we may restrict to estimators of the form $\delta(C) = f(\sum_i C_i)$ for some function $f$. Now, observe $\sum_i C_i \sim p + \Bin(r,1/2)$ and recall $\tau(\Delta) = (p-q)/n$. Then, the minimax risk $r_n^*$ is given by
\begin{align*}
    & \inf_f \sup_{\Delta} \E_\Delta( f(\sum_i C_i) - \tau(\Delta))^2 
    =  \inf_f \sup_{\Delta} \E( f(p + \Bin(r,1/2)) - (p-q)/n)^2 \\
    = & \inf_f \sup_{p,q,r} \E( f(p + \Bin(r,1/2)) - (p-q)/n)^2,
\end{align*} 
which establishes the first claim.\\

For the second claim, apply Lemma \ref{lemma:existence_of_lfp} with $\Theta = \cl{I}_n, \cl{X} = \set{0,1,\dots,n}, P_\theta = \cl{L}(X) = p + \Bin(r,1/2).$ Note the estimand is $(p-q)/n \in [-1,1]$ and we may restrict to estimators $f$ in $[-1,1].$ We obtain a least favorable prior $\pi_n^*$ on $\cl{I}_n$ such that
\[
\rho_n^* = \inf_f \sup_{p,q,r \in \cl{I}_n}  \E\left( f(X) - \frac{p-q}{n}\right)^2 = \inf_\delta \E_{\pi_n^*}\E\left( \delta(X) - \frac{p-q}{n}\right)^2.
\]
The estimator achieving the infimum on the right-hand side is the Bayes estimator $\delta^*$ defined by $\delta^*(X) = \E[(p-q)/n \mid X].$ Call the resulting Bayes risk $R_{\pi_n^*}(\delta^*)$. Lemma \ref{lemma:existence_of_lfp} also establishes the existence of a minimax estimator $f_n^*$, so that $\sup_{\cl{I}_n} \E\left( f^*_n(X) - \frac{p-q}{n}\right)^2 = \rho_n^*$. Then, its average risk under $\pi_n^*,$ denoted $R_{\pi_n^*}(f_n^*)$, is bounded above by its maximum risk and therefore by $\rho_n^*$:
\(
R_{\pi_n^*}(\delta^*) \leq R_{\pi_n^*}(f_n^*) \leq \rho_n^* = R_{\pi_n^*}(\delta^*).
\)
By projection properties of Bayes estimators, we conclude that $\E_{\pi_n^*}\E[(f^*_n(X) - \delta^*(X))^2] = 0$, or more explicitly 
\begin{equation}
    \sum_{x = 0}^n \Prob_{\pi_n^*}(X = x) (f^*_n(x) - \delta^*(x))^2 = 0.
\end{equation}

In the remainder of the proof, we will show that $\Prob_{\pi_n^*}(X = x) > 0$ for all $x$. If this is true, it is immediate that $f^*_n(x) = \delta^*(x)$ for all $x,$ which shows uniqueness, equality to the posterior mean of $(p-q)/n$ under $\pi_n^*$, and therefore admissibility.

Towards this goal, let $M_n^*$ be the unnormalized minimax risks
\begin{equation}
\label{eq:unnormalized_minimax_risks}
M_n^* = \inf_f \sup_{\cl{I}_n} \E\left[\left(f\left(X\right) - (p-q) \right)^2\right],
\end{equation}
so that $M^*_n = n^2\rho_n^*$. For the edge case, let $M_0^* = 0.$ Suppose for the sake of contradiction that $\Prob_{\pi_n^*}(X = x) = 0$ for some $x \in \set{0,1,\dots,n}$. Because the support of $X$ is exactly $\set{p,p+1,\dots,p+r}$, it must follow that either 1) $p \geq x + 1$, or 2) $p + r  + 1\leq x$, which is equivalent to $q \geq n - x + 1.$ Let $\cl{I}_n^{(1)}$ denote the subset of $\cl{I}_n$ such that the first case holds, and $\cl{I}_n^{(2)}$ the subset for the second case. Therefore, $\pi_n^*$ is supported on $\cl{I}_n^{(1)} \cup \cl{I}_n^{(2)}$. \\

We can explicitly compute the minimax risk on this union space $\cl{I}_n^{(1)} \cup \cl{I}_n^{(2)}$ in terms of two easier subproblems:
\begin{align*}
 & \inf_f \sup_{\cl{I}^{(1)}_n \cup \cl{I}^{(2)}_n} \E\left[ (f(X) - (p-q))^2 \right] \\
 = & \max\left( \inf_f\sup_{\cl{I}^{(1)}_n} \E\left[ (f(X) - (p-q))^2 \right], \inf_f \sup_{\cl{I}^{(2)}_n} \E\left[ (f(X) - (p-q))^2 \right] \right).
\end{align*}
The last equality holds due to the following crucial observation: on $\cl{I}_n^{(1)}$, the possible support of $X$ is $\set{x+1,\dots,n}$ while on $\cl{I}_n^{(2)}$ it is $\set{0,\dots,x-1}$. Because these are disjoint, one can simply combine the minimax estimators on the individual spaces to obtain a minimax estimator on the union.

For any $(p,q,r) \in \cl{I}_n^{(1)},$ define $(p',q',r')$ by $p' = p - x - 1, q' = q, r' = r$. This transformation defines a bijection between $\cl{I}_n^{(1)}$ and $\cl{I}_{n - x - 1}.$ Then
\begin{align*}
    & \inf_f\sup_{\cl{I}^{(1)}_n} \E\left[ (f(X) - (p-q))^2 \right] 
     = \inf_f\sup_{\cl{I}^{(1)}_n} \E\left[ (f(p' + \Bin(r,1/2) + (x+1))-(x+1) - (p'-q'))^2 \right] \\
     & =  \inf_{\tilde{f}} \sup_{(p',q',r') \in \cl{I}_{n-x-1}} \E\left[ (\tilde{f}(p' + \Bin(r',1/2)) - (p'-q'))^2 \right].
\end{align*}

Therefore, the minimax risk over $\cl{I}_n^{(1)}$ is equal to $M_{n-x-1}^*$. The same argument applies to $\cl{I}_n^{(2)}$, using the bijection which maps $(p,q,r) \in \cl{I}_n^{(2)}$ to $(p',q',r') \in \cl{I}_{x-1}$ via $p' = p, q' = q - n + x - 1, r' = r.$ This shows the minimax risk over 
$\cl{I}_n^{(2)}$ is exactly $M_{x-1}^*$. Combining these facts, we conclude 
\begin{equation}
\inf_f \sup_{\cl{I}^{(1)}_n \cup \cl{I}^{(2)}_n} \E\left[ (f(X) - (p-q))^2 \right] = \max(M_{x-1}^*, M^*_{n - x - 1}).
\end{equation}
If either $x-1$ or $n-x-1$ falls outside the set $\set{0,1,\dots,n}$, omit the corresponding term from the maximum. Because $\pi_n^*$ is supported on $\cl{I}_n^{(1)} \cup \cl{I}_n^{(2)}$, we conclude that the Bayes risk of $\pi_n^*$ is at most the minimax risk on $\cl{I}_n^{(1)} \cup \cl{I}_n^{(2)}$, yielding
\begin{equation}
\label{eq:unnormalized_risk_inequality_easier_subproblem}
M_n^* \leq \max(M_{x-1}^*, M^*_{n - x - 1}).
\end{equation}
Lemma \ref{lemma:unnormalized_minimax_risks_strictly_increasing}  shows that $M_n^*$ is strictly increasing in $n$, which is a contradiction to \eqref{eq:unnormalized_risk_inequality_easier_subproblem}.

\end{proof}

\begin{lemma}
\label{lemma:unnormalized_minimax_risks_strictly_increasing}
Let $M_n^*$ be the unnormalized risks defined in Eq. \eqref{eq:unnormalized_minimax_risks}. Then
for any $m,l$, $ M^*_m + M^*_l \leq M^*_{m+l}$. In particular, the sequence $(M_n^*)_{n\ge 1}$ is strictly increasing.
\end{lemma}
\begin{proof}
As in the proof of Theorem \ref{thm:simplified_model_minimax_risk}, there exist least favorable priors $\pi_n^*$ for each $n$ with Bayes risks equal to $M_n^*$.
Draw random variables 
$(P^{(1)},Q^{(1)},R^{(1)})$ 
from $\pi_m^*$ and independently $(P^{(2)},Q^{(2)},R^{(2)}) \sim \pi_l^*$. Generate $X^{(1)} \sim P^{(1)} + \Bin(R^{(1)},1/2), X^{(2)} \sim P^{(2)} + \Bin(R^{(2)},1/2)$, conditionally independently. Set $(P,Q,R) = (P^{(1)} + P^{(2)},Q^{(1)} + Q^{(2)},R^{(1)} + R^{(2)})$, which belongs to $\cl{I}_{m + l}.$ Now, define
\[
X := X^{(1)} + X^{(2)} \sim P + \Bin(R,1/2). 
\]
Consider the model with prior $G$ on $\cl{I}_{m+l}$ given by the law of $(P,Q,R)$, with observation $X$ and estimand $P-Q$. Its Bayes risk is the expected conditional variance $\E_G[\Var_G(P - Q \mid X)]$. This is at most $M^*_{m + l}$, since Theorem \ref{thm:simplified_model_minimax_risk} shows $M^*_{m+l}$ is equal to the Bayes risk of the least favorable prior $\pi^*_{m+l}$ on $\cl{I}_{m+l}$ under the same observation model. By the standard data processing inequality $\E[\Var(X \mid \cl{G}_1)] \leq \E[\Var(X \mid \cl{G}_2)]$ for any random variable $X$ and sigma-algebras $\cl{G}_1 \supseteq \cl{G}_2$, we have
\begin{align*}
    \E_G[\Var_G(P - Q \mid X)] & \geq \E_G[\Var_G(P - Q \mid X^{(1)},X^{(2)})] \\
    & = \E_G[\Var_G(P^{(1)} - Q^{(1)} \mid X^{(1)},X^{(2)})] + \E_G[\Var_G(P^{(2)} - Q^{(2)} \mid X^{(1)},X^{(2)})] \\
    & = \E_G[\Var_G(P^{(1)} - Q^{(1)} \mid X^{(1)})] + \E_G[\Var_G(P^{(2)} - Q^{(2)} \mid X^{(2)})] 
     = M^*_m + M^*_l. 
\end{align*}
This proves the superadditivity property. Finally, observe that $M_1^* > 0$. This completes the proof.
\end{proof}

\begin{proof}[Proof of Theorem \ref{thm:minimax_risk_reduction}]
\noindent \textit{Lower Bound.} Consider any prior $\pi(G)$ on potential outcomes $\cl{P}$ generated by first sampling parameters $\Delta \in \set{-1,0,1}^n$ from a distribution $G$, and independently sampling the potential outcomes as in \eqref{eq:bayes_model}. Throughout this proof, note that $\Delta_i = Y_i(1) - Y_i(0)$ and also define 
\(
C_i = A_i Y_i + (1-A_i)(1-Y_i).
\) 
Regardless of $A,$ the law of $C$ is that described in \eqref{eq:C_definition}. Indeed, when $\Delta_i = 1,$ we have $C_i = A_iY_i(1) + (1 - A_i)(1 - Y_i(0)) = 1.$ When $\Delta_i = -1$, we have $C_i = A_i \cdot 0 + (1 - A_i)\cdot 0 = 0.$ Finally, when $(Y_i(1),Y_i(0)) = (1,1), C_i = A_i$ and when $(Y_i(1),Y_i(0)) = (0,0), C_i = 1 - A_i$. Therefore, conditional on $\Delta_i = 0, C_i \sim \textsf{Ber}(1/2)$ independently of $A$. In particular, $\Prob(C,A \mid \Delta) = \Prob(C \mid \Delta)\Prob_{\scr{A}}(A)$.

By Bayes' Rule, we have
\begin{equation}
    \Prob(\Delta \mid C,A) = \Prob(\Delta \mid C),
\end{equation}
so $A$ is ancillary for the parameter $\Delta$. Then, for any design $\scr{A}$ and estimator $\hat \tau$,
\begin{align*}
&\sup_{\cl{P}} \E_{A,Y}\left[(\hat \tau(A,Y) - \tau)^2 \right]  \geq \sup_{\pi(G)} \E_{\cl{P} \sim \pi(G)}\E\left[(\hat \tau(A,Y) - \tau)^2 \right] \\ 
& \geq \sup_{\pi(G)} \E_{\cl{P},A,Y}(\E[\tau \mid C,A] - \tau)^2  
 = \sup_{\pi(G)}  \E_{\cl{P},Y}(\E[\tau \mid C] - \tau)^2.
\end{align*}
The latter quantity is the worst-case Bayes risk of estimating $\tau$ under the hierarchical model $\Delta \sim G$ and $\cl{P}$ constructed as in Equation \eqref{eq:bayes_model}. This is exactly the structure of the reduced model. Apply Lemma \ref{lemma:existence_of_lfp} by taking $\Theta = \set{-1,0,1}^n, \cl{X} = \set{0,1}^n$ and $P_\theta$ the law of $C$. We conclude there is a least favorable prior $G_{\textsf{LF}}$ on $\Delta$ such that the Bayes risk under $G_{\textsf{LF}}$ is equal to $\rho_n^*$. Note the restriction to $[-1,1]$ in Lemma \ref{lemma:existence_of_lfp} does not increase risk and is thus immaterial for the reduced model. Then
\(
\sup_{\cl{P}} \E_{A,Y}\left[(\hat \tau(A,Y) - \tau)^2 \right] \geq r_n^*.
\)
By Theorem \ref{thm:simplified_model_minimax_risk}, $r_n^* = \rho_n^*$.\\

\noindent \textit{Upper Bound.} Consider \BRE \ with treatment probability $1/2$ so $A_i \sim \textsf{Ber}(1/2)$ i.i.d., as well as $C_i = A_i Y_i + (1-A_i)(1-Y_i)$. As we checked above, 
\begin{equation}
\label{eq:transformed_outcomes_C_cases}
C_i =
\begin{cases}
1 & \text{ if } (Y_i(1),Y_i(0)) = (1,0) \\
0 & \text{ if } (Y_i(1),Y_i(0)) = (0,1) \\
1 - A_i & \text{ if } (Y_i(1),Y_i(0)) = (0,0) \\
A_i & \text{ if } (Y_i(1),Y_i(0)) = (1,1).
\end{cases}
\end{equation}
Under Bernoulli randomization, this is exactly the law of the intermediate model \eqref{eq:C_definition}. Then, $\sum_i C_i \sim p + \Bin(r,1/2)$ and so $f^*_n(\sum C_i)$ has maximum risk exactly $\rho_n^* = r_n^*$.\\

\noindent \textit{Admissibility.} Suppose the procedure $(\BRE,\Opt)$ were dominated by some other procedure $(\scr{A},\hat \theta)$. Reparametrize the estimator $\hat \theta$ to be a function of the assignment vector $A \sim \scr{A}$ and $C.$ Domination implies for every potential outcome configuration $\cl{P}$,
\begin{equation}
\label{eq:risk_domination_of_alternative_procedure}
\textstyle \E\left[\left(f^*_n\left(\sum_i C_i \right) - \tau \right)^2 \right] \geq \E_{\scr{A}}[(\hat \theta(A,C) - \tau)^2], 
\end{equation}
with strict inequality for one configuration $\cl{P}_0$. 
Let $\Delta_0$ be the vector of individual treatment effects under $\cl{P}_0.$ 

For any fixed vector $\Delta \in \set{-1,0,1}^n$, let $P_\Delta$ be the prior on potential outcomes from \eqref{eq:bayes_model}, and let $p,q,r$ denote the number of $1,-1,0$ ITEs respectively in $\Delta$. Averaging \eqref{eq:risk_domination_of_alternative_procedure} over $P_\Delta$ yields for any $\Delta$
\[
\textstyle \E_{P_\Delta} \E\left[\left(f^*_n\left(\sum_i C_i \right) - \tau \right)^2 \right] \geq \E_{P_\Delta} \E_{\scr{A}}[(\hat \theta(A,C) - \tau)^2],
\]
with strict inequality for $\Delta_0$. The left hand side is equal to $ \E\left[\left(f^*_n\left(\sum_i C_i \right) - \tau \right)^2 \right]$, because $P_\Delta$ is supported on configurations with the same values of $(p,q,r)$ as $\Delta$ and the law of $\sum_i C_i$ only depends on $\Delta$ through $(p,q,r)$. Define the estimator $\hat \tau:\set{0,1}^n \rightarrow \bR$ by $\hat \tau(c) = \E_{\scr{A}}\hat \theta(A,c)$ for the reduced model. It was shown in the first part of the proof that for any fixed $\Delta$, $C \indep A \mid \Delta$. With this fact and Jensen's inequality, we obtain
\begin{equation}
\label{eq:risk_domination_of_alternative_procedure_II}
\textstyle \E\left[\left(f^*_n\left(\sum_i C_i \right) - \tau \right)^2 \right] \geq \E_{P_{\Delta}}[(\hat \tau(C) - \tau)^2],
\end{equation}
for all $\Delta$ with strict inequality for $\Delta_0$. Under $P_\Delta$, the data $C$ has exactly the same law as in the reduced model \eqref{eq:C_definition}. Therefore \eqref{eq:risk_domination_of_alternative_procedure_II} can be interpreted strictly as a statement about the reduced model of Section \ref{sec:reduction_to_simplified_model}. By a similar argument in the proof of Theorem \ref{thm:simplified_model_minimax_risk}, we may symmetrize $\hat \tau$ to obtain another estimator $\bar \tau (\sum_i C_i)$ so that
\begin{equation}
\textstyle \E\left[\left(f^*_n\left(\sum_i C_i \right) - \tau \right)^2 \right] \geq \E_{P_{\Delta}}[(\bar \tau(\sum_i C_i) - \tau)^2],
\end{equation}
and strictly smaller risk at $\Delta_0.$ In particular, 
\begin{equation}
\textstyle \E_{p,q,r}\left[\left(f^*_n\left(p + \Bin(r,1/2) \right) - (p-q)/n \right)^2 \right] \geq \E_{p,q,r}[(\bar \tau(p + \Bin(r,1/2)) - (p-q)/n)^2],
\end{equation}
for all $(p,q,r) \in \cl{I}_n$ and strict inequality at the $(p,q,r)$ values of $\Delta_0$. However, this implies $\bar{\tau}$ is minimax and therefore equal to $f_n^*$ by the uniqueness property of Theorem \ref{thm:simplified_model_minimax_risk}. This contradicts the purported strict risk dominance at $\Delta_0$. 

\end{proof}

\begin{proof}[Proof of Theorem \ref{thm:minimax_risk_reduction_bounded_case}]
We will first show the lower bound. Suppose $Y_i(a) \in [L,U], a = 0,1$. Define the recentered potential outcomes
\(
\tilde Y_i(a) := \frac{Y_i(a) - L}{U - L}
\)
which are in $[0,1]$ and let $\tilde Y$ be the same transformation applied to the observed data $Y$. This transformation is bijective. Furthermore, let $\tilde \tau = \tau / (U - L)$, which also equals $\frac{1}{n}\sum_{i=1}^n \tilde Y_i(1) - \tilde Y_i(0)$. For each estimator $\hat \tau,$ define 
\[
\hat \tau_{\Bin}(a,y) = \frac{1}{U-L}\hat \tau\left(a,(U-L)y+L\right),
\]
so $\hat \tau(A,Y) = (U-L)\hat \tau_{\Bin}(A,\tilde{Y}).$
Then
\begin{align*}
    \inf_{A,\hat \tau}\sup_{[L,U]^{2n}}\E\left[\left(\hat \tau(A,Y) - \tau\right)^2 \right] 
     & = (U-L)^2 \inf_{A,\hat \tau}\sup_{\tilde Y(a) \in [0,1]^{2n}}\E\left[\left(\hat \tau_{\Bin}(A,\tilde Y) - \tilde \tau\right)^2 \right] \\
    & = (U-L)^2 \inf_{A,\hat \tau}\sup_{\tilde Y(a) \in [0,1]^{2n}}\E\left[\left(\hat \tau(A,\tilde Y) - \tilde \tau\right)^2 \right] \\
    & \geq (U-L)^2 \inf_{A,\hat \tau}\sup_{\tilde Y(a) \in \set{0,1}^{2n}}\E\left[\left(\hat \tau(A,\tilde Y) - \tilde \tau\right)^2 \right] \\
     & = (U-L)^2r_n^*.
\end{align*}

For the upper bound, it suffices to handle the case where $[L,U] = [0,1]$ because the previous transformation $(Y_i(a) - L)/(U - L)$ may then be applied. Fix $\set{Y_i(1),Y_i(0)}_{i=1}^n \in [0,1]^{2n}$ and consider again $C_i = A_i Y_i(1) + (1-A_i)(1-Y_i(0)) \in [0,1]$. Let $f_n^*$ be the minimax optimal estimator in the simplified experiment, from Theorem \ref{thm:simplified_model_minimax_risk}. We claim the procedure with \BRE \ and the estimator $g^*$ defined by
\[
g^*(C) := \E\left[f^*_n\left(\sum_{i=1}^n B_i\right) \mid C \right], \quad B_i \stackrel{ind}{\sim} \Ber(C_i)
\]
has maximum risk $\rho_n^*$. We claim $g^*$ can equivalently be written as 
\begin{equation}
\label{eq:helper_1}
g^*(C) = \E_G\left[ f^*_n\left(\sum_{i=1}^n T_i \right) \mid A\right],
\end{equation}
where $T_i = A_iZ_i(1) + (1-A_i)(1-Z_i(0))$ and $G$ is a prior on the potential outcomes $(Z_i(1),Z_i(0))_{i=1}^n$ given by $Z_i(a) \stackrel{ind}{\sim} \Ber(Y_i(a))$ for all $i=1,\dots,n, a \in \set{0,1}.$ Under $G$, it is easy to see that $T_i \mid A_i = 1 \sim \Ber(Y_i(1)) = \Ber(C_i),$ and $T_i \mid A_i = 0 \sim \Ber(1 - Y_i(0)) = \Ber(C_i).$ Therefore, conditional on $A,$ $T_i$ are independent with laws $\Ber(C_i)$, which establishes \eqref{eq:helper_1}. 

Now, $\tau = \frac{1}{n}\sum_{i=1}^n (Y_i(1) - Y_i(0)) = \E_G\left[\frac{1}{n}\sum_{i=1}^n Z_i(1) - Z_i(0) \right]$. Let $\tau_Z := \frac{1}{n}\sum_{i=1}^n Z_i(1) - Z_i(0)$. Because the design $A$ is independent of the outcomes $Z,$ we conclude that $\tau = \E[\tau_Z \mid A]$. Expanding the risk of $g^*$ yields
\begin{align*}
    &\E\left[\left(g^*(C_1,\dots,C_n) - \tau \right)^2\right]  = \E\left[\left(\E_G \left[f^*_n\left(\sum T_i \right) - \tau_Z \mid A\right]\right)^2\right] \\
    & \leq \E\left[\E_G \left(f^*_n\left(\sum T_i \right) - \tau_Z \right)^2\right] 
     = \E_G\left[ \E \left(f^*_n\left(\sum T_i \right) - \tau_Z \right)^2\right].
\end{align*}
The inner quantity, for any fixed realization of $Z$, is just the risk of $f_n^*$ on the binary potential outcome configuration $(Z_i(1),Z_i(0))_i$. This is upper bounded by $\rho_n^* = r_n^*$, as desired.
\end{proof}

\section{Proofs for Section \ref{sec:second_order_risk}}
 
Following the notation in Section \ref{sec:second_order_risk}, consider the model where we observe data $X = \theta + \sum_{i=1}^r \e_i$ with unknown parameters $(\theta,r) \in \Theta_n$ and wish to estimate $\theta/n$. Recall the function
$\phi_A$ such that for all real $t$,
\begin{equation}
\label{eq:variational_minimizer}
\phi_A(t) := c \Ai\left( \frac{|t| - K_A}{4^{1/3}}\right),
\end{equation}
where $K_A = -4^{1/3}a_1'$, with $a_1'$ being the largest negative zero of the derivative of the Airy function \cite{vallee2010airy}.
We will choose $c$ such that $\int_0^\infty \phi_A(t)^2 dt = 1$. Finally, define $h_A := -2\phi'_A/\phi_A$. Proposition \ref{prop:airy_solution_properties} will show $\phi_A > 0$ and so $h_A$ is well-defined. Figure \ref{fig:airy_ground_state_and_shrinkage} visualizes these functions alongside the Airy function for convenience.

\begin{figure}[H]
    \centering
    \includegraphics[width=\linewidth]{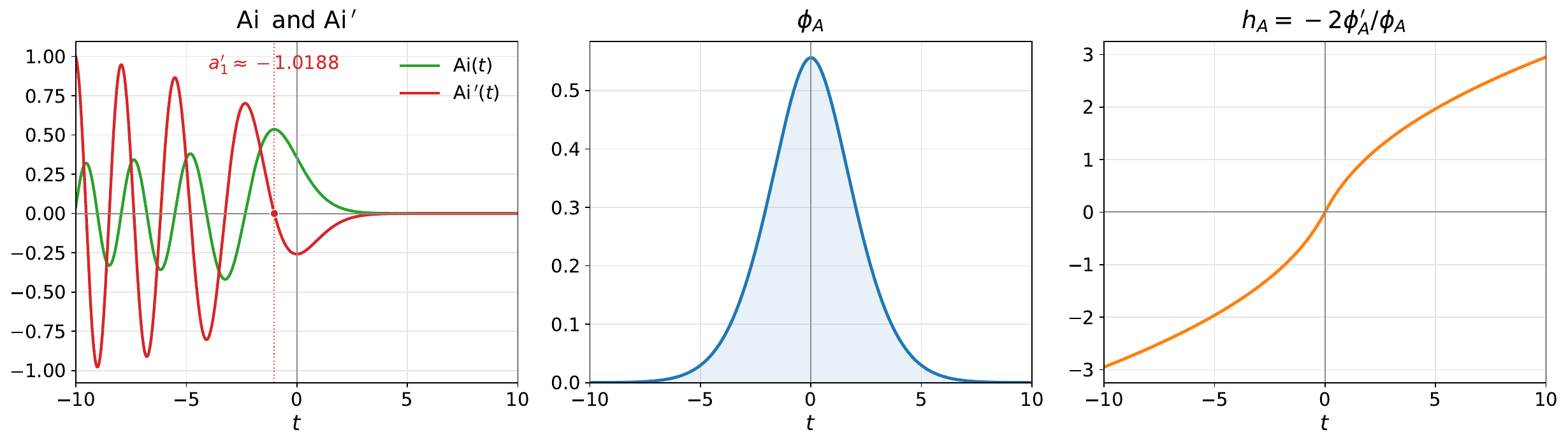}
    \caption{Plots of the Airy function and its derivative along with $\phi_A, h_A$. }
    \label{fig:airy_ground_state_and_shrinkage}
\end{figure}

Theorem \ref{thm:second_order_minimax_risk_behavior} will follow as a direct consequence of the following results, which provide the explicit second-order minimax estimator and a sequence of asymptotically least favorable priors that yields the lower bound. Our proofs do not justify the variational heuristic in Section \ref{sec:second_order_risk}, which was only used to derive the form of the second-order minimax estimator. 

\begin{theorem}
\label{thm:second_order_minimax_risk_behavior_detailed_upper_bound}
The estimator $\hat \tau_n^A$ defined by $\hat \tau_n^A(X) := \frac{X}{n} - \frac{1}{n^{2/3}}h_A\left( \frac{X}{n^{2/3}}\right)$ satisfies the risk bound
\[
\sup_{(\theta,r) \in \Theta_n} \E[(\hat \tau_n^A(X) - \theta/n)^2] \leq \frac{1}{n} - \frac{K_A}{n^{4/3}} + o_n(n^{-4/3}).
\]
\end{theorem}

\begin{theorem}
\label{thm:second_order_minimax_risk_behavior_detailed_lower_bound}
The priors $\Pi_n$ on $\Theta_n$ defined by
\begin{align*}
    & \Pi_n((0,n)) \propto \phi_A(0)^2, \quad
     \Pi_n((k,n-k)) = \Pi_n((-k,n-k)) \propto \frac{1}{2}\phi_A\left(\frac{k}{n^{2/3}}\right)^2, \quad 1\leq k \leq n
\end{align*}
have Bayes risks lower bounded by $\frac{1}{n} - \frac{K_A}{n^{4/3}} + o_n(n^{-4/3}).$ 
\end{theorem}

We will establish the upper bound first and handle the lower bound in the next subsection. The proof relies on the analysis of the function in \eqref{eq:variational_minimizer}.

\begin{proposition}
\label{prop:airy_solution_properties} 
The function $\phi_A$ is strictly positive, even, twice continuously differentiable, and  satisfies the differential equation $-4\phi_A''(t) + |t|\phi_A(t) = K_A \phi_A(t)$. Moreover, $\int_0^\infty \set{4(\phi_A')^2 + t\phi_A^2} dt = K_A.$ Finally, with $Q(t) = \phi_A(t)^2$ and some positive constants $C,c$,
\begin{equation}
Q(t) + |Q'(t)| + \frac{Q'(t)^2}{Q(t)} \leq Ce^{-c|t|^{3/2}}.
\end{equation}
\end{proposition}
\begin{proof}
Standard properties of the Airy function show that $\Ai$ is $C^2$, $\Ai(z) > 0, \Ai'(z) \leq 0$ for all $z \geq a_1'$ \cite[\S9.9]{NIST_DLMF}. In addition, the following asymptotics hold as $x \rightarrow \infty$ \cite[Equations 9.7.5, 9.7.6]{NIST_DLMF}, \cite[pp. 392, 413]{olver1997asymptotics}:
\begin{equation}
\label{eq:airy_asymptotics}
\Ai(x)=\frac{x^{-1/4}}{2\sqrt{\pi}}e^{-2x^{3/2}/3}
\left(1+O(x^{-3/2})\right),\qquad
\Ai'(x)=-\frac{x^{1/4}}{2\sqrt{\pi}}e^{-2x^{3/2}/3}
\left(1+O(x^{-3/2})\right).
\end{equation}
Write $\phi_A(t) = c\Ai\left(\frac{|t|}{4^{1/3}} + a_1'\right)$. It follows immediately that $\phi_A$ is strictly positive and even. Computing the first derivative of $\phi_A$,
\[
\phi_A'(t) = 
\begin{cases}
& c4^{-1/3}\Ai'\left(4^{-1/3}t + a_1' \right), \quad t \geq 0 \\
& -c4^{-1/3}\Ai'\left(-4^{-1/3}t + a_1' \right), \quad t \leq 0. 
\end{cases}
\]
The limit from the right and left is equal to $c4^{-1/3}\Ai'(a_1')$ and $-c4^{-1/3}\Ai'(a_1')$ respectively, which are both zero by definition of $a_1'$. Thus $\phi_A \in C^1(\bR)$. Computing the second derivative gives
\[
\phi_A''(t) = c4^{-2/3}\Ai''\left(4^{-1/3}|t| + a_1' \right),
\]
which is continuous on $\mathbb{R}$.
Therefore, $\phi_A$ is $C^2$. Set $z := (t - K_A)/4^{1/3}$. By the definition of the Airy function, $\Ai''(z) = z\Ai(z)$ for $z \in \bR$. Therefore, $\phi_A''(t) - \frac{t - K_A}{4}\phi_A(t) = 0$ for $t \geq 0$. Applying the same argument for $t \leq 0$ and combining yields
\begin{equation}
\label{eq:eigenvalue_equation_for_phi_A}
-4\phi_A''(t) + |t|\phi_A(t) = K_A\phi_A(t), \forall t \in \bR.
\end{equation}
Next, multiply \eqref{eq:eigenvalue_equation_for_phi_A} by $\phi_A(t)$ and integrate over $\bR$. Because $\phi_A(t), \phi_A'(t) \rightarrow 0$ as $|t| \rightarrow \infty$ integrating by parts yields
\begin{align*}
    & -4\int_\bR \phi_A''(t) \phi_A(t) dt + \int_\bR |t|\phi_A(t)^2 dt = K_A \int_\bR \phi_A(t)^2 dt \\
    \Longrightarrow & 4\int_\bR \phi_A'(t)^2 dt + \int_\bR |t|\phi_A(t)^2 dt  = K_A \int_\bR \phi_A(t)^2 dt.
\end{align*}
Evenness and the normalization $\int_0^\infty \phi_A(t)^2 dt = 1$ gives 
\(
\int_0^\infty \set{4\phi_A'(t)^2 + t\phi_A^2 } dt = K_A.
\)
For the last claim, the Airy asymptotics in Eq. \eqref{eq:airy_asymptotics}, after absorbing polynomial factors into the exponential and enlarging constants if necessary, yield
\(
|\phi_A(t)| + |\phi'_A(t)| \leq Ce^{-c|t|^{3/2}}
\)
for some constants $C,c > 0.$ With $Q(t) = \phi_A(t)^2, Q'(t) = 2\phi_A(t)\phi_A'(t), Q'(t)^2/Q(t) = 4\phi_A'(t)^2$, the desired bound is immediate.

\end{proof}

\begin{proposition}
\label{prop:second_order_perturbation_h_properties} 
The function  $h_A$ is odd and positive for $t > 0.$ Moreover, the following are true:
\begin{enumerate}
    \item $h_A$ satisfies the differential equation $h_A(t)^2 - 2h_A'(t) - |t| = -K_A.$
    \item $h_A$ satisfies the estimates $|h_A(t)| \leq K(1 + \sqrt{|t|})$ and $|h_A'(t)| \leq K/(1 + \sqrt{|t|})$.
    \item $h_A^2,h_A'$ are globally Lipschitz continuous.
\end{enumerate}
\end{proposition}

\begin{proof}
Firstly, $h_A$ is odd because $\phi_A$ is even. The properties $\Ai(z) > 0, \Ai'(z) < 0$ for all $z > a_1'$ show that $h_A$ is positive for $t > 0.$ Direct differentiation of $h_A$ yields
\[
h_A' = -2\left(\frac{\phi_A''}{\phi_A} - \left(\frac{\phi_A'}{\phi_A}\right)^2 \right) = -\frac{2\phi_A''}{\phi_A} + \frac{1}{2}h_A^2.
\]
Thus, $h_A^2(t) - 2h_A'(t) = 4\frac{\phi_A''(t)}{\phi_A(t)}$. From the eigenvalue equation in Proposition \ref{prop:airy_solution_properties}, we obtain $h_A^2(t) - 2h_A'(t) - |t| = -K_A$, which is the first claim. For the second claim, recall the Airy asymptotics from Eq. \eqref{eq:airy_asymptotics}. Dividing immediately gives
\[
m(x) := -\frac{\Ai'(x)}{\Ai(x)} = x^{1/2} + O(x^{-1}), x \rightarrow \infty.
\]
Writing $h_A(t) = \frac{2}{4^{1/3}}m\left(\frac{t - K_A}{4^{1/3}}\right), t \geq 0$ gives the estimate $|h_A(t)| \lesssim 1 + \sqrt{|t|}.$ Taking the derivative $m'$ and using the Airy differential equation, we have
\begin{align*}
    m'(x) & = \frac{-\Ai''(x)\Ai(x) + \Ai'(x)^2}{\Ai(x)^2} 
    = -x + m(x)^2  = O(x^{-1/2}),
\end{align*}
which gives the estimate on $h_A'(t).$ Finally, the derivative of $h_A^2$ is $2h_Ah_A'$. The estimates just proved show that $|2h_A(t)h_A'(t)| = O(1),$ which proves Lipschitzness of $h_A^2$. From the differential equation $h_A(t)^2 - 2h_A'(t) - |t| = -K_A,$ rearranging and taking the derivative yields $h_A''(t) = h_A(t)h_A'(t) - \frac{1}{2}\text{sgn}(t), t \neq 0$, so $h_A'(t)$ is Lipschitz on each half-line. Continuity and triangle inequality shows the Lipschitz bound for any interval crossing $0.$
\end{proof}

With this, we will prove the upper bound on the minimax risk in Theorem \ref{thm:second_order_minimax_risk_behavior_detailed_upper_bound}, which elaborates on the risk expansion described in Section \ref{sec:second_order_risk}.

\begin{proof}[Proof of Theorem \ref{thm:second_order_minimax_risk_behavior_detailed_upper_bound}]
Let $\tilde \theta = \theta/n^{2/3}.$ We expand the risk of the estimator $\hat \tau_n^A$ as
\begin{align*}
    & \E[(\hat\tau_n^A(X) - \theta/n)^2] 
   = \ \E\left[\left(\frac{U_r}{n} - \frac{1}{n^{2/3}} h_A\left( \frac{X}{n^{2/3}}\right)\right)^2 \right] \\
    &= \  \frac{r}{n^2} + \frac{1}{n^{4/3}}\E h_A\left(\tilde \theta + \frac{U_r}{n^{2/3}} \right)^2 - \frac{2}{n^{5/3}}\E\left[ U_r h_A\left(\tilde \theta + \frac{U_r}{n^{2/3}} \right) \right].
\end{align*}

Applying Lemma \ref{lemma:stein_rademacher_sums},
\[
\E\left[ U_r h_A\left(\tilde \theta + \frac{U_r}{n^{2/3}} \right) \right] = \frac{r}{2n^{2/3}} \E \int_{-1}^1 h_A'\left(\tilde \theta + \frac{U_{r-1} + s}{n^{2/3}} \right) ds.
\]
Using $\E|U_r| \leq \sqrt{r} \leq \sqrt{n}$ and the Lipschitzness of $h_A',h_A^2$ from Proposition \ref{prop:second_order_perturbation_h_properties}, we have
\begin{align*}
    & \E h_A\left(\tilde \theta + \frac{U_r}{n^{2/3}} \right)^2 = h_A(\tilde \theta)^2 + O(n^{-1/6}), \\
    & \E \int_{-1}^1 h_A'\left(\tilde \theta + \frac{U_{r-1} + s}{n^{2/3}} \right) ds = 2h_A'(\tilde \theta) + O(n^{-1/6}).
\end{align*}
Assembling,
\begin{align*}
\E[(\hat\tau_n^A(X) - \theta/n)^2] & = \frac{r}{n^2} + \frac{h_A(\tilde \theta)^2 - 2rh_A'(\tilde \theta)/n}{n^{4/3}} + O(n^{-3/2}) \\
& = \frac{r}{n^2}\left(1 - \frac{2h_A'(\tilde \theta)}{n^{1/3}} \right) + \frac{h_A(\tilde \theta)^2}{n^{4/3}} + O(n^{-3/2}),
\end{align*}
which holds uniformly over $(\theta,r) \in \Theta_{n}.$ Proposition \ref{prop:second_order_perturbation_h_properties} shows $h_A'$ is bounded, so $1 - \frac{2h_A'(\tilde \theta)}{n^{1/3}} > 0$ for large enough $n$. Taking the supremum over $r$  over the bound $r \leq n - |\theta|$ yields
\[
\sup_{ \Theta_n} \E[(\hat\tau_n^A(X) - \theta/n)^2] = \frac{1}{n} + \sup_{\theta}\left( \frac{h_A(\tilde \theta)^2 -2h_A'(\tilde \theta) - |\tilde \theta|}{n^{4/3}} + \frac{2|\tilde \theta| h_A'(\tilde \theta)}{n^{5/3}} \right) + O(n^{-3/2}).
\]
By Proposition \ref{prop:second_order_perturbation_h_properties}, $h_A(t)^2 - 2h'_A(t) - |t| = - K_A$ for all $t.$ Thus, 
\[
\sup_{ \Theta_n} \E[(\hat\tau_n^A(X) - \theta/n)^2] = \frac{1}{n} - \frac{K_A}{n^{4/3}} + \frac{2\sup_{\theta} |\tilde \theta| h_A'(\tilde \theta)}{n^{5/3}}  + O(n^{-3/2}).
\]
By the same proposition, $ |th_A'(t)| \leq K\sqrt{|t|}$. Combining this and $\tilde \theta = O(n^{1/3})$ shows $\sup_{\theta} |\tilde \theta| h_A'(\tilde \theta) = O(n^{1/6})$. Therefore $\frac{2\sup_{\theta} |\tilde \theta| h_A'(\tilde \theta)}{n^{5/3}} = O(n^{-3/2})$. This shows the upper bound.
\end{proof}

\subsection{Proof of the Lower Bound.}

\begin{lemma}[Discrete Van Trees Inequality]
\label{lemma:discrete_van_trees}
Let $\theta \in \bb{Z} \cap [-n,n]$ and $r = n - |\theta|$. For any positive probability vector $\rho = (\rho_0,\dots,\rho_n)$, let $G$ be the symmetric prior given by $\Prob_G(\theta = 0) = \rho_0$ and
\[
\Prob_G(\theta = k) = \Prob_G(\theta = -k) = \rho_k/2, \quad 1 \leq k \leq n.
\]
Suppose we observe $X = \theta + \sum_{i=1}^{r} \e_i$ with $\theta \sim G$. Then for every estimator $\delta$,
\begin{equation}
\label{eq:van_trees_inequality}
\E\left[\left(\delta(X) - \theta/n \right)^2 \right] \geq \frac{(1-\rho_0)^2}{n^2 I_n(\rho)},
\end{equation}
where
\begin{equation}
\label{eq:discrete_information}
I_n(\rho) := \sum_{k=0}^n \frac{(\rho_{k+1} - \rho_k)^2}{\rho_k} + \sum_{k=0}^{n-1} \frac{\rho_{k+1}^2}{\rho_k(n-k)},
\end{equation}
and $\rho_{n+1}$ is interpreted as $0.$
\end{lemma}
\begin{proof}
Because $(\theta,X) \stackrel{(d)}{=} (-\theta,-X)$, the reflected estimator $\delta^R$ defined by $\delta^R(x) = -\delta(-x)$ has the same average risk as that of $\delta$. By convexity, the estimator $(\delta + \delta^R)/2$ has no larger average risk than $\delta$. Therefore, we may prove the inequality \eqref{eq:van_trees_inequality} for odd rules. For $\theta \geq 0$, define $Y := (n-X)/2$, so that $Y \sim \Bin(n - \theta,1/2)$. Define $P_k(j) := \Prob_k(Y = j) =  2^{-(n-k)}\binom{n-k}{j}$ and $\tilde{\delta}(x) = \delta(n - 2x).$ By oddness, the average risk of $\delta$ under $G$ can be written
\(
    \sum_{k = 0}^n \rho_k \E_k\left[\left(\tilde{\delta}(Y) - \frac{k}{n} \right)^2 \right].
\)
This is the average risk of the estimator $\tilde \delta$ when generating $\tilde{\theta} \in \set{0,\dots,n}$ from the probability vector $\rho$ and then $Y \sim \Bin(n - \tilde \theta,1/2),$ so
\[
 \sum_{k = 0}^n \rho_k \E_k\left[\left(\tilde{\delta}(Y) - \frac{k}{n} \right)^2 \right] = \E_\rho\left[\left(\tilde \delta(Y) - \frac{\tilde \theta}{n} \right)^2\right].
\]
Next, define $S_k(j) := \frac{\rho_{k+1}P_{k+1}(j) - \rho_{k}P_{k}(j)}{\rho_k P_k(j)}$ for all $0 \leq k \leq n$, with $\rho_{n+1}$ interpreted as $0$. Then, 
\begin{align*}
\E_\rho\left[\left(\tilde \delta(Y) - \frac{\tilde \theta}{n} \right) S_{\tilde \theta}(Y)\right] & = \sum_{j,k} \left(\tilde \delta(j) - \frac{k}{n}\right)\left(\rho_{k+1}P_{k+1}(j) - \rho_{k}P_{k}(j) \right) \\
& = \sum_{k = 0}^{n-1} \rho_{k+1}\E_{k+1}\left[ \tilde \delta(Y) - \frac{k}{n} \right] - \sum_{k = 0}^n \rho_{k}\E_{k}\left[ \tilde \delta(Y) - \frac{k}{n} \right] \\
& = \frac{1}{n}\sum_{k=1}^n \rho_k - \rho_0 \E_0 \tilde \delta(Y) 
 = \frac{1 - \rho_0}{n}.
\end{align*}
The last line follows from oddness of $\delta$ and symmetry of the law of $X$ when $\tilde \theta = 0$, so $\E_0 \tilde \delta(Y) = 0.$ Applying Cauchy-Schwarz yields
\[
\left(\frac{1 - \rho_0}{n}\right)^2 \leq \E_\rho\left[\left(\tilde \delta(Y) - \frac{\tilde \theta}{n} \right)^2\right] \E_\rho\left[ S_{\tilde \theta}(Y)^2\right].
\]
It remains to show that $\E_\rho\left[ S_{\tilde \theta}(Y)^2\right] = I_n(\rho).$ Expanding the definition,
\begin{align*}
\E_\rho\left[ S_{\tilde \theta}(Y)^2\right] & = \sum_{k=0}^{n} \rho_k \E\left(\frac{\rho_{k+1}}{\rho_k} \frac{P_{k+1}(Y)}{P_k(Y)} - 1\right)^2 
 = \sum_{k=0}^{n-1} \rho_k \E\left(\frac{2\rho_{k+1}}{\rho_k} \frac{n - k - Y}{n - k} - 1\right)^2 + \rho_n \\
& = \rho_n + \sum_{k=0}^{n-1} \left(\frac{4\rho_{k+1}^2}{\rho_k} \E\left(1 - \frac{Y}{n-k} \right)^2 - 2\rho_{k+1} + \rho_k\right) 
 = \sum_{k=0}^{n-1} \frac{\rho^2_{k+1}}{\rho_k(n-k)} + \sum_{k=0}^n \frac{(\rho_{k+1} - \rho_k)^2}{\rho_k},
\end{align*}
as desired.
\end{proof}

\begin{proof}[Proof of Theorem \ref{thm:second_order_minimax_risk_behavior_detailed_lower_bound}]
Let $\e_n = n^{-2/3}$ and $Q(x) = \phi_A(x)^2$. The standard Airy asymptotics \eqref{eq:airy_asymptotics} imply that $Q(x)$ and its derivative are $O(e^{-cx^{3/2}}), x \geq 1$. Consider the probability vector $\rho_n = (\rho_{n,0},\dots,\rho_{n,n})$ defined by 
\(
\rho_{n,k} := \frac{Q(\e_n k)}{Z_n}
\)
with $Z_n$ the normalization constant $\sum_{k=0}^n Q(\e_nk)$.  The main goal is to apply the discrete Van Trees Inequality in Lemma \ref{lemma:discrete_van_trees}. We first claim that
\begin{equation}
\label{eq:e_nZ_n_term}
\e_n Z_n \rightarrow \int_0^\infty Q(t) dt = 1.
\end{equation}
\sloppy By convergence of Riemann sums, we know $\e_n \sum_{k = 0}^\infty Q(\e_n k) \rightarrow \int_0^\infty Q(t) dt = 1$. But $\left|\e_n (Z_n - \sum_{k = 0}^\infty Q(\e_n k))\right| \leq \e_n \sum_{k=n+1}^\infty Q(\e_n k) \lesssim \e_n \sum_{k=n+1}^\infty \exp(-ck^{3/2}/n) \rightarrow 0.$ This establishes \eqref{eq:e_nZ_n_term}; as an immediate result, $\rho_{n,0} = O(n^{-2/3}).$\\

\noindent \textit{First Information Term.} Now, the first term in the information $I_n$ from Eq. \eqref{eq:discrete_information} is
\[
I_{1,n} := \frac{1}{Z_n} \sum_{k=0}^{n-1}\frac{(Q(\e_n(k+1)) - Q(\e_n k))^2}{Q(\e_n k)}  + \rho_{n,n}.
\]
Since $\rho_{n,n}$ is exponentially small due to \eqref{eq:airy_asymptotics},
\[
\frac{I_{1,n}}{\e_n^2} = \frac{1}{\e_nZ_n} \left[\sum_{k=0}^{n-1}\e_n \frac{(Q(\e_n(k+1)) - Q(\e_n k))^2}{\e_n^2Q(\e_n k)} \right] + o(1).
\]
If we can show
\begin{equation}
\label{eq:handling_I_1_term}
  \frac{1}{\e_nZ_n} \left[\sum_{k=0}^{n-1}\e_n \frac{(Q(\e_n(k+1)) - Q(\e_n k))^2}{\e_n^2Q(\e_n k)} \right]  = \int_0^\infty \frac{Q'(t)^2}{Q(t)} dt + o_n(1),
\end{equation}
it would immediately follow that 
\begin{equation}
\label{eq:first_information_term_analysis}
    I_{1,n} = \frac{4\int_0^\infty \phi'_A(t)^2 dt}{n^{4/3}} + o_n(n^{-4/3}).
\end{equation}
Towards \eqref{eq:handling_I_1_term} note that $\e_n Z_n \rightarrow 1$ by our previous argument, so we can focus on the bracketed term. Let $a_{n,k} := \frac{\rho_{n,k+1}}{\rho_{n,k}} = \frac{Q(\e_n (k+1))}{Q(\e_n k)}, 0\leq k \leq n-1$. Since $Q'/Q = -h_A$ on the positive real line, 
\begin{equation}
a_{n,k} = \exp\left(-\int_{\e_nk}^{\e_n(k+1)} h_A(t) dt \right).
\end{equation}
Since $h_A(t) \geq 0$ for $t \geq 0, a_{n,k} \in [0,1]$.  Then
\[
\sum_{k=0}^{n-1}\e_n \frac{(Q(\e_n(k+1)) - Q(\e_n k))^2}{\e_n^2Q(\e_n k)} = \sum_{k=0}^{n-1}\e_n Q(\e_n k)\left(\frac{1 - a_{n,k}}{\e_n} \right)^2.
\]
Note that $h_A$ is uniformly continuous on any compact interval, which implies $\frac{1}{\e_n}\int_{\e_nk}^{\e_n(k+1)} h_A(t) dt = h_A(\e_nk) + o(1)$ uniformly for $\e_n k \leq M$. Therefore, $(1-a_{n,k})/\e_n = h_A(\e_n k) + o(1)$ uniformly for $\e_n k \leq M$. Breaking up the sum into parts, we have 
\begin{align*}
& \sum_{k=0}^{n-1}\e_n Q(\e_n k)\left(\frac{1 - a_{n,k}}{\e_n} \right)^2 
= \sum_{k: \e_n k < M}\e_n Q(\e_n k)\left(-\frac{Q'(\e_n k)}{Q(\e_n k)} + o(1)\right)^2 + R_n \\
= & \sum_{k: \e_n k < M}\e_n \frac{Q'(\e_n k)^2}{Q(\e_n k)} + R_n + o(1) 
=  \int_0^M \frac{Q'(t)^2}{Q(t)} dt + R_n + o_n(1),
\end{align*}
with a remainder term $R_n$ that satisfies
\begin{align*}
 R_n & := \sum_{k:\e_nk > M} \e_n Q(\e_n k)\left(\frac{1 - a_{n,k}}{\e_n} \right)^2 
    \leq \sum_{k:\e_nk > M} \e_n Q(\e_n k)\left(\frac{\int_{\e_nk}^{\e_n(k+1)} h_A(t) dt}{\e_n} \right)^2 \\ 
    & \lesssim \sum_{k:\e_nk > M} \e_n Q(\e_n k)\e_n k 
     = \int_M^\infty tQ(t) dt + o(1).
\end{align*}
The second to last line follows from the growth bound $h_A(t) \lesssim K(1 + \sqrt{t})$. Because $Q(t)$ has exponential tails, $\int_0^\infty tQ(t) dt < \infty$ and so taking $n \rightarrow \infty$ then $M \uparrow \infty$ yields
\[
\sum_{k=0}^{n-1}\e_n \frac{(Q(\e_n(k+1)) - Q(\e_n k))^2}{\e_n^2Q(\e_n k)} = \int_0^\infty \frac{Q'(t)^2}{Q(t)} dt + o_n(1),
\]
as desired.\\

\noindent \textit{Second Information Term.} Next, we handle the second term in $I_n$ arising from the likelihood, which we will rewrite 
\(
I_{2,n} := \sum_{k=0}^{n-1} \frac{\rho_{n,k}}{n - k} \cdot \left( \frac{\rho_{n,k+1}}{\rho_{n,k}}\right)^2.
\)
The growth bound on $h_A(t)$ in Proposition \ref{prop:second_order_perturbation_h_properties} and the elementary inequality $1 - e^{-2x} \leq 2x$ yields
\[
1 - a_{n,k}^2 \leq K\e_n(1 + \sqrt{\e_n k}).
\]
Note $\sum_{k=0}^n \rho_{n,k}(1 + \sqrt{\e_n k}) \rightarrow 1 + \int_0^\infty \sqrt{x} Q(x) dx < \infty.$ Therefore,  $\sum_{k=0}^{n-1} \rho_{n,k}|a_{n,k}^2 - 1| = O(\e_n).$ Use this estimate to replace the first $n/2$ terms in $I_{2,n}$ with $\rho_{n,k}/(n-k)$, with an error of $2n^{-1}O(\e_n) = O(n^{-5/3})$. The sum of the terms where $k > n/2$ is exponentially small:
\begin{equation}
\label{eq:prior_tail_exponentiall_small}
\sum_{k=n/2}^{n-1} \frac{\rho_{n,k}a_{n,k}^2}{n-k} \leq \sum_{k = n/2} \rho_{n,k} \leq \frac{\e_n \sum_{k = n/2}^{n-1} Q(\e_n k)}{\e_nZ_n} \leq \frac{\e_n \sum_{k = n/2}^{n-1} \exp(-ck^{3/2}/n)}{\e_nZ_n} = O(e^{-c'\sqrt{n}}).
\end{equation}
Thus,
$   I_{2,n} = \sum_{k=0}^{n-1} \frac{\rho_{n,k}}{n-k} + o(n^{-4/3})$.
For $k \leq n/2$, it is not hard to show that 
$\left|\frac{1}{n - k} - \frac{1}{n} -\frac{k}{n^2}\right| \leq \frac{Ck^2}{n^3}$
for some absolute constant $C.$ Then by Riemann sum convergence,
\begin{align*}
    \sum_{k=0}^{n-1} \frac{\rho_{n,k}}{n-k} & = \sum_{k=0}^{n-1} \frac{\rho_{n,k}}{n} + \sum_{k=0}^{n-1} \frac{k\rho_{n,k}}{n^2} + O\left(\sum_{k=0}^{n-1} \frac{k^2\rho_{n,k}}{n^3}\right) \\
    & = \frac{1}{n} +  n^{-2}\left(n^{2/3} \int_0^\infty tQ(t) dt + o(n^{2/3}) \right) + O(n^{-5/3}) 
     = \frac{1}{n} +   \frac{\int_0^\infty tQ(t) dt}{n^{4/3}} + o(n^{-4/3}).
\end{align*}

Using again the prior tail estimate in \eqref{eq:prior_tail_exponentiall_small} to handle the terms $k > n/2$, we conclude 
\begin{equation}
\label{eq:second_information_term_analysis}
I_{2,n} = \frac{1}{n} + \frac{\int_0^\infty t\phi_A(t)^2 dt}{n^{4/3}} + o(n^{-4/3}).
\end{equation}
\noindent \textit{Combining.} Adding the two information terms in \eqref{eq:second_information_term_analysis}, \eqref{eq:first_information_term_analysis} and applying Proposition \ref{prop:airy_solution_properties} gives
\begin{align*}
    I_n(\rho) & =I_{1,n} + I_{2,n} 
    = \frac{1}{n} + \frac{\int_0^\infty \set{t\phi_A(t)^2 + 4\phi_A'(t)^2} dt}{n^{4/3}} + o(n^{-4/3}) 
     = \frac{1}{n} + \frac{K_A}{n^{4/3}} + o(n^{-4/3}).
\end{align*}
Van Trees' Inequality lower bounds the Bayes risk $B_n$ under $\Pi_n$ by $\frac{(1 - \rho_{n,0})^2}{n^2I_n(\rho)}$. By our previous analysis we have
\[
B_n \geq \frac{1}{n}\left(1 + O(n^{-2/3}) \right) \left(1 - K_A n^{-1/3} + o_n(n^{-1/3})\right) = \frac{1}{n} - \frac{K_A}{n^{4/3}} + o(n^{-4/3}),
\]
as desired.
\end{proof}

\section{Proofs for Section \ref{sec:comparison_to_standard_methods}}

\begin{proof}[Proof of Prop. \ref{prop:maximum_risks}]
Throughout this proof, let $p,q,r_1,r_0$ be the number of potential outcome pairs $(Y_i(1),Y_i(0))$ which equal $(1,0),(0,1),(1,1),(0,0)$ respectively and set $r = r_1 + r_0$. Consider the first claim regarding \CRE \ with the difference in means estimator \eqref{eq:dim}. The variance formula for difference-in-means in a completely randomized experiment with $n_1$ treated units is well-known. Defining $\mu_1 := \frac{1}{n}\sum_{i=1}^n Y_i(1), \mu_0 := \frac{1}{n}\sum_{i=1}^n Y_i(0)$ and $\bar{\mu}$ their average, it is given by
\begin{equation}
\label{eq:risk_of_CRE_DIM}
\Var(\hat{\tau}_{\DIM}) = \frac{1}{n_1}S_1^2 + \frac{1}{n-n_1}S_0^2 - \frac{1}{n}S_\tau^2,
\end{equation}
with $S_1^2 = \frac{1}{n-1}\sum_{i=1}^n (Y_i(1) - \mu_1)^2, S_0^2 = \frac{1}{n-1}\sum_{i=1}^n (Y_i(0) - \mu_0)^2$ and $S_\tau^2 = \frac{1}{n-1}\sum_{i=1}^n [Y_i(1) - Y_i(0) - (\mu_1 - \mu_0)]^2$. When $n_1 = n/2$, the largest possible value of $S_1^2$ and $S_0^2$ is $n/(4(n-1))$, when the finite populations are equally split between $0$ and $1$.  Thus under the \CRE, 
\(
\Var(\hat{\tau}_{\DIM}) \leq \frac{2}{n}S_1^2 + \frac{2}{n}S_0^2 \leq \frac{1}{n-1}.
\)
This upper bound is attained by the configuration where $r_0 = r_1 = n/2$.

Next, consider the procedure $(\BRE,\DIM)$. Note this procedure is biased because of the edge-case handling. Let $K$ be the number of treated units. By conditioning on the number of treated units $K = n_1$, the experiment is a completely randomized experiment with $n_1$ units treated. The law of total variance gives
\begin{align*}
    \E(\hat \tau_{\DIM} - \tau)^2 & = \E\left[\Var(\hat \tau_{\DIM} - \tau \mid K)\right] + \tau^2 2^{1-n}. \\
    & = \E\left[\Var(\hat \tau_{\DIM} \mid K)\right] + \tau^2 2^{1-n}. \\
    & = \E\left[ \frac{1}{K} S_1^2 + \frac{1}{n - K}S_0^2 - \frac{1}{n}S_\tau^2; K \notin \set{0,n}\right] + \tau^22^{1-n}.
\end{align*}
Letting $a_n := \E[K^{-1}; K \notin \set{0,n}],$ we have
\begin{equation}
\label{eq:risk_of_BRE_DIM}
\E(\hat \tau_{\DIM} - \tau)^2 = a_n (S_1^2 + S_0^2) - \frac{(1-2^{1-n})}{n}S_\tau^2 + 2^{-n+1}\tau^2.
\end{equation}
 Because $Y_i(1), Y_i(0) \in [0,1]$, we have $(n-1)S_1^2 = \sum_{i=1}^n (Y_i(1) - \mu_1)^2 = \sum_{i=1}^n Y_i(1)^2 - n\mu_1^2 \leq n\mu_1(1 - \mu_1)$. The same holds for $S_0^2$. Then
\begin{align*}
\E(\hat \tau_{\DIM} - \tau)^2  & \leq \frac{n}{n-1}a_n\left( \mu_1(1-\mu_1) + \mu_0(1-\mu_0)\right) + 2^{1-n}\tau^2 \\
& = \frac{n}{n-1}a_n\left(2\bar{\mu} - 2\bar{\mu}^2 -\frac{1}{2}\tau^2\right) + 2^{1-n}\tau^2 
 \leq \frac{n}{2(n-1)}a_n(1-\tau^2) + 2^{1-n}\tau^2.
\end{align*}
It is easy to show that $na_n/(2n-2) \geq 2^{1-n}$ for all $n \geq 2$. Therefore, $\E(\hat \tau_{\DIM} - \tau)^2  \leq na_n/(2n-2)$. The configuration where $r_0 = r_1 = n/2$ also satisfies $\E(\hat \tau_{\DIM} - \tau)^2  = na_n/(2n-2)$. This proves the maximum risk of (\BRE, \DIM) is exactly $na_n/(2n-2)$. Now, we compute the asymptotic expansion of this quantity. By Lemma \ref{lemma:binomial_inverse_moment}, $a_n = \frac{2}{n} + \frac{2}{n^2} + O(n^{-3})$ and $n/(n-1) = 1 + \frac{1}{n} + \frac{1}{n^2} + O(n^{-3})$. Therefore, 

\[
\sup_{\cl{P}} \E(\hat \tau_{\DIM} - \tau)^2  = \frac{1}{n} + \frac{2}{n^2} + O(n^{-3}).
\]

Thirdly, we prove the claim that in a standard \BRE, the Horvitz-Thompson estimator \eqref{eq:horvitz_thompson} has maximum risk $4/n$. Under the \BRE, note that \eqref{eq:horvitz_thompson} is unbiased. It suffices to compute its variance. 
\begin{align*}
    \Var(\hat \tau_{\HT}) & = \frac{4}{n^2} \sum_{i=1}^n \Var(A_i Y_i(1) - (1-A_i)Y_i(0)) \\
    & = \frac{4}{n^2} \sum_{i=1}^n \Var(A_i(Y_i(1) + Y_i(0)) - Y_i(0)) 
    = \frac{1}{n^2} \sum_{i=1}^n (Y_i(1) + Y_i(0))^2.
\end{align*}
This is maximized when $r_1 = n$, yielding $4/n.$

For the procedure $(\BRE,\cHT)$, we compute the risk explicitly. Note that
\begin{align*}
\hat \tau_{\cHT} - \tau & = \frac{1}{n}\sum_{i=1}^n \left[A_i(2Y_i(1) - 1) - (1-A_i)(2Y_i(0) - 1)  - \left(Y_i(1) - Y_i(0)\right)\right].
\end{align*}
When $A_i = 1$, the summand is $(Y_i(1) + Y_i(0) - 1)$. When $A_i = 0$, it is $-(Y_i(1) + Y_i(0) - 1)$. Therefore, 
\begin{equation}
\label{eq:representation_for_cHT}
\hat \tau_{\cHT} - \tau = \frac{1}{n}\sum_{i=1}^n  (2A_i - 1)(Y_i(1) + Y_i(0) - 1).
\end{equation}
This representation, and the related one for $\hat \tau_{\HT}$, has been used previously in the literature \cite{kallus2018optimal,harshaw2024balancing}. From this representation, it is easy to see unbiasedness. Computing the variance, we obtain
\[
\Var(\hat \tau_{\cHT}) = \frac{1}{n^2} \sum_{i=1}^n (Y_i(1) + Y_i(0) - 1)^2,
\]
which is maximized whenever $(Y_i(1),Y_i(0)) \in \set{(1,1),(0,0)}$ for all $i.$
\end{proof}

\begin{proof}[Proof of Theorem \ref{thm:admissibility_of_procedures}]
\noindent \textit{Admissibility of (\CRE,\DIM).} Recall that under the \CRE, the estimators $\hat \tau_{\DIM}, \hat \tau_{\cHT}, \hat \tau_{\HT}$ are all equivalent. The key insight is to note that for any potential outcome configuration $\cl{P}$ such that $Y_i(1) + Y_i(0) = 1$ for all $i$, the risk is exactly zero. This is immediate from our derivation of $\hat \tau_{\cHT} - \tau$ in \eqref{eq:representation_for_cHT}. Any potentially dominating procedure $(\scr{A}, \hat \theta)$ must therefore also have zero risk on such configurations. Fix any vectors $y \in \set{0,1}^n, a \in \set{0,1}^n$ which represent possible realizations of data and treatment respectively. Define the configuration of potential outcomes $\cl{P}_{a,y}$ by
\[
(Y_i(1),Y_i(0)) = 
\begin{cases}
(y_i,1-y_i) & \text{ if } a_i = 1 \\
(1-y_i,y_i) & \text{ if } a_i = 0.
\end{cases}
\]
By domination, we have $\E_{A \sim \scr{A}} \left[(\hat \theta(A,Y) - \tau)^2 \right] = 0.$ In particular, for any $a$ such that $\Prob_{\scr{A}}(A = a) > 0, \hat \theta(a,Y(a)) = \tau.$ Under $\cl{P}_{a,y}$, the realized data $Y(a)$ under the assignment $a$ is exactly the vector $y$. Moreover, the SATE $\tau$ for such a configuration is equal to $\frac{1}{n}\sum_{i=1}^n (2a_i - 1)(2y_i - 1).$ Therefore,
\begin{equation}
    \hat \theta(a,y) = \frac{1}{n}\sum_{i=1}^n (2a_i - 1)(2y_i - 1),
\end{equation}
for all choices of $y$ and $a$ in the support of $\scr{A}$. Therefore, $\hat \theta(A,Y) = \frac{1}{n}\sum_{i=1}^n (2A_i - 1)(2Y_i - 1)$ on the support of the design, which is exactly the $\hat \tau_{\cHT}$ estimator after some manipulation. Any dominating procedure must therefore use $\hat \tau_{\cHT}$. We will proceed to show that no choice of design can dominate $\CRE$ once the estimator is fixed to be $\hat \tau_{\cHT}$.

By \eqref{eq:representation_for_cHT}, the risk for $\hat \tau_{\cHT}$ is equal to $\E[u^\top \tilde{A}\tilde{A}^\top u]/n^2$ where $\tilde{A} = 2A - 1$ and $u = Y(1) + Y(0) - 1$. For the CRE, the design second-moment matrix is easily analyzed and is given by
\[
\Sigma_{\CRE} = \E[\tilde{A}\tilde{A}^\top] = \frac{n}{n-1}I_n - \frac{1}{n-1}\mathbf{1}\mathbf{1}^\top.
\]
Let $\Sigma_{\scr{A}}$ be the analogous second-moment matrix under the dominating design $\scr{A}$. Domination therefore requires
\(
u^\top (\Sigma_{\scr{A}} - \Sigma_{\CRE}) u \leq 0,
\)
for all $u \in \set{-1,0,1}^n$. Taking the choice of vectors $u = e_i,-e_i$ for each $i$ and the fact that $\tilde{A}_i^2 = 1$ shows that the diagonals must be equal. The choices $u = e_i + e_j, e_i - e_j$ show that the off-diagonal entries must be equal. Therefore, the risks are exactly equal and strict domination of $(\scr{A},\hat \theta)$ at some configuration fails. We conclude $(\CRE,\DIM)$ is admissible. Note the $\CRE$ design was only used insofar as $\hat{\tau}_{\cHT} = \hat{\tau}_{\DIM} = \hat{\tau}_{\HT}$ for this design. \\

\noindent \textit{Admissibility of $(\BRE,\cHT)$.} The proof is the same as the previous claim. \\

\noindent \textit{Domination of $(\BRE,\HT)$.} We will show that $(\CRE,\DIM)$ dominates $(\BRE, \HT).$ The usual Horvitz-Thompson estimator admits a representation akin to \eqref{eq:representation_for_cHT}. By the same logic, it can be checked that 
\[
\hat \tau_{\HT} - \tau = \frac{1}{n} \sum_{i=1}^n (2A_i - 1)(Y_i(1) + Y_i(0)).
\]
Since $\hat \tau_{\DIM} = \hat \tau_{\HT}$ for the \CRE, it suffices to show $u^\top \Sigma_{\BRE} u \geq u^\top \Sigma_{\CRE} u$ where $u = Y(1) + Y(0) \in \set{0,1,2}^n$. This reduces to checking 
\(
u^\top\left(I_n - \frac{n}{n-1}I_n + \frac{1}{n-1}\mathbf{1}\mathbf{1}^\top \right)u \geq 0,
\)
or $(\sum u_i)^2 \geq \sum u_i^2$ for all $u \in \set{0,1,2}^n$ which is clearly true by nonnegativity of the entries. It is easy to exhibit one configuration where the risk domination is strict. \\

\noindent \textit{Domination of $(\BRE,\DIM).$} We computed the risks of $(\CRE,\DIM),(\BRE,\DIM)$ in Equations \eqref{eq:risk_of_CRE_DIM}, \eqref{eq:risk_of_BRE_DIM}. Subtracting the two gives
\[
R_n(\BRE,\DIM) - R_n(\CRE,\DIM) = \left(a_n - \frac{2}{n}\right)(S_1^2 + S_0^2) + \frac{2^{1-n}}{n}S_\tau^2 + 2^{-n+1}\tau^2.
\]
Recall that $a_n = \E[K^{-1};K\notin \set{0,n}]$ and $K \sim \Bin(n,1/2).$ We claim $a_n > 2/n$, which would show each of the terms in the risk difference is nonnegative. By pairing $1/k$ and $1/(n-k),$
\begin{align*}
2a_n & = 2^{-n}\sum_{k=1}^{n-1} \binom{n}{k}\left(\frac{1}{k} + \frac{1}{n-k} \right) 
 = 2^{-n}\sum_{k=1}^{n-1} \binom{n}{k}\left(\frac{4}{n} + \frac{(n - 2k)^2}{nk(n-k)}\right) \\
& \geq \frac{4}{n}\left(1 - 2^{1-n} \right) + 2^{-n}\sum_{k=1}^{n-1} \binom{n}{k}\frac{(n - 2k)^2}{nk(n-k)} 
 \geq \frac{4}{n}\left(1 - 2^{1-n} \right) + 2^{1-n} \frac{(n-2)^2}{(n-1)},
\end{align*}
by keeping the terms with $k = 1,n-1$. Since $(n-2)^2/(n-1) > 4/n$ for $n \geq 4$, the proof is complete. The configuration where $r_1 = n/2$ and $r_0 = n/2$ shows the risk difference is strictly greater than zero.
\end{proof}

\section{Proofs for Section \ref{sec:survey_experiments}}

\begin{proof}[Proof of Theorem \ref{thm:survey_experiment_extension}]

Throughout this proof recall $C_i := A_iY_i + (1-A_i)(1-Y_i)$.\\

\textit{Lower Bound.} As in the proof of Theorems \ref{thm:simplified_model_minimax_risk} and \ref{thm:minimax_risk_reduction}, we will lower bound the minimax risk by the Bayes risk of an appropriate prior. Fix a least-favorable prior $G$ on $\cl{I}_N$ for the problem \eqref{eq:survey_experiment_reduced_model}, which exists by Lemma \ref{lemma:existence_of_lfp}. Generate $(Y_i(1),Y_i(0))_{i=1}^N$ as follows. First, sample $(p,q,r) \sim G$ and consider a vector $\Delta$ uniformly chosen amongst all vectors in $\set{-1,0,1}^N$ with $p$ entries $+1$, $q$ entries $-1$, and $r$ entries $0$. Finally, generate $(Y_i(1),Y_i(0))$ as in Equation \eqref{eq:bayes_model}: $(Y_i(1),Y_i(0)) = (1,0)$ if $\Delta_i = 1$, $(Y_i(1),Y_i(0)) = (0,1)$ if $\Delta_i = -1$, and $(Y_i(1),Y_i(0)) = (0,0)$ or $(1,1)$ independently with probability $1/2$ if $\Delta_i = 0$. The joint law $\pi(G)$ on the potential outcomes 
\begin{equation}
\label{eq:population_potential_outcomes}
\Big((Y_1(1),Y_1(0)),\dots, (Y_N(1),Y_N(0))\Big)
\end{equation}
is exchangeable and so any size $n$ subset of this vector has the same law. Let $S$ be the sample, a size $n$ subset of $[N]$. For such any fixed size-$n$ subset $s$, we define
\begin{align*}
A_s = (A_i)_{i \in s}, \qquad Y_s = (Y_i)_{i \in s}, \qquad C_s = (C_i)_{i \in s}.
\end{align*}
Define $Y_{S}(a), a \in \set{0,1}$ similarly. Since every size-$n$ subset of \eqref{eq:population_potential_outcomes} has the same law, the joint law of $(Y_s(1), Y_s(0))$ is the same for every $s$. Following the argument of Theorem \ref{thm:minimax_risk_reduction}, the law of $C_{1:N} \mid p,q,r$ does not depend on the experimental design $\scr{A}$.

Let $\bb{Q}_{p,q,r}(x)$ be the probability of observing $X = x$ in the model \eqref{eq:survey_experiment_reduced_model}. We claim that $\Prob(C_s = c \mid S = s, A_s = a,p,q,r)$ is equal to $\bb{Q}_{p,q,r}(\sum_{i \in s} c_i)$, divided by $\binom{n}{\sum_{i\in s} c_i}$. To see this, $C_i \mid \Delta$ is given by \eqref{eq:C_definition}. Given $p,q,r$, the binary vector $C_{1:N}$ is uniformly chosen among all binary vectors with $K := p+\Bin(r,1/2)$ many 1's. Then $\Prob(C_s = c \mid S =s, A_s = a,p,q,r)$ is determined by the number of these 1's which are contained in $s$: it is $\Prob(\textsf{Hypergeom}(N,K,n) = \sum_{i\in s} c_i)$, divided by the number of ways of arranging the ones in $s$. Putting these together, the joint law of $S, A, C$ factors as
\[
\Prob(S = s, A_{S} = a, C_{S} = c \mid p,q,r) = \Prob(S = s) \Prob(A_s = a \mid S = s) \frac{\textstyle \bb{Q}_{p,q,r}\left(\sum_{i \in s} c_i\right)}{\textstyle \binom{n}{\sum_{i\in s} c_i}}.
\]

Because $\Prob(S = s)$ and $\Prob(A_s = a \mid S = s)$ do not depend on $p,q,r$, these terms cancel when applying Bayes' rule. Therefore, the posterior distribution of $p - q$ given the observed data $(S,A_{S},C_{S})$ is equal to the posterior distribution of $p-q$ given $\sum_{i \in S} C_i$ which has the same law as $\eqref{eq:survey_experiment_reduced_model}$. Finally,
\begin{align*}
    \sup_{\cl{P}} \E[(\hat \tau - \tau_{\PATE})^2] & \geq \E_{\cl{P} \sim \pi(G)} \E[(\hat \tau - \tau_{\PATE})^2] \\
    & \geq \E_{\cl{P} \sim \pi(G)} \E[( \E[\tau_{\PATE} \mid S,A_{S},C_{S}] - \tau_{\PATE})^2] \\
    & = \E_{G,X}[(\E[(p-q)/N \mid X] - (p-q)/N)^2],
\end{align*}
and the last line is the Bayes risk for the model \eqref{eq:survey_experiment_reduced_model} under the prior $G$ for $(p,q,r).$ Using the fact that $G$ is least favorable completes the lower bound. \\

\textit{Upper Bound.} Let $\Delta_i,i=1,\dots,N$ be the ITEs in the population, and $p,q,r$ be the number of units such that $\Delta_i = 1,-1,0$ respectively. With these variables, $\tau_{\text{PATE}} = (p-q)/N$. Consider the purported minimax optimal procedure. Under independent Bernoulli randomization, we can first assign treatment indicators before taking a sample. Thus, let $A_i \simiid \Ber(1/2)$ for $i=1,\dots,N$. Then $C_i$ have exactly the law described in Eq. \eqref{eq:transformed_outcomes_C_cases}. The number of indices $i \in [N]$ such that $C_i = 1$ is distributed as $p + \Bin(r,1/2).$ Under simple random sampling of $n$ units, $\sum_{i: Z_i = 1} A_iY_i + (1-A_i)(1-Y_i)$ therefore has the law described in \eqref{eq:survey_experiment_reduced_model}. Then $g^*(\sum_{i: Z_i = 1} A_iY_i + (1-A_i)(1-Y_i))$ has maximum risk $\rho_{N,n}^*$. \\

\textit{Limiting Risk.} Let $c_n = \sqrt{n}/(\sqrt{n}+1)$. We will establish the bounds
\begin{equation}
\label{eq:survey_sampling_risk_bounds}
    \frac{(N - n)(N + \sqrt{n})}{N^2} \frac{1}{(1 + \sqrt{n})^2} \leq \rho_{N,n}^* \leq \frac{1}{(1 + \sqrt{n})^2}.
\end{equation}
For the upper bound, consider the estimator $f$ with $f(x) = c_n\left( \frac{2x - n}{n}\right)$. We will calculate its maximum risk by evaluating the bias and variance, using standard formulas for the mean and variance of a hypergeometric distribution. Let $Z := \frac{2X - n}{n}$ and let $\tau = (p-q)/N, m_2 = (p+q)/N$. We can calculate
\begin{align*}
    \E[Z] & = \left(2\frac{\E K}{N} - 1 \right) = \left(\frac{2p + r}{N} - 1 \right) = \tau, \qquad
    \Var(Z)  = \frac{1}{n}\left(1 - \frac{(n-1)}{N-1}m_2 - \frac{N-n}{N-1}\tau^2\right).
\end{align*}
The variance derivation is left to Lemma \ref{lemma:variance_calculation}. Observe that $c_n^2/n = (1 - c_n)^2 = \frac{1}{(1+ \sqrt{n})^2}$. Then
\begin{align*}
    &\E[(f(X) - \tau)^2] = (\E[f(X)] - \tau)^2 + \Var(f(X)) 
     = (c_n \E[Z] - \tau)^2 + c_n^2 \Var(Z) \\
    & = \frac{1}{(1+ \sqrt{n})^2}\left(1 - \frac{n-1}{N-1} m_2 + \frac{n-1}{N-1}\tau^2 \right) 
     = \frac{1}{(1+ \sqrt{n})^2}\left(1 - \frac{n-1}{N-1}\left(m_2 - \tau^2 \right) \right) 
     \leq \frac{1}{(1+ \sqrt{n})^2},
\end{align*}
where the last line follows because $m_2 - \tau^2 \geq 0 \Leftrightarrow N(p+q) \geq (p-q)^2$ and $|p-q| \leq \min(p + q, N)$. This establishes the upper bound in \eqref{eq:survey_sampling_risk_bounds}.\\

For the lower bound, sample $U \sim \Beta(\sqrt{n}/2,\sqrt{n}/2)$ and conditional on $U$, generate $p \sim \Bin(N,U), q = N - p, r = 0.$ The minimax risk $\rho_{N,n}^*$ is lower bounded by the Bayes risk of estimating $(p-q)/N = 2p/N - 1,$ which is thus
\begin{equation}
\label{eq:Bayes_risk}
\frac{4}{N^2}\E[\Var( p \mid X)].
\end{equation}
In this model, notice that $K = p$. Conditional on $U$, we may equivalently express the model as $p = \sum_{i=1}^N B_i, X = \sum_{i=1}^N Z_i B_i$, where $B_1,\dots,B_N \simiid \Ber(U)$ and $Z_1,\dots,Z_N$ are the indicators of simple random sampling without replacement. The vectors $B,Z$ are independent. Therefore the following properties are true: 1) $p \mid U \sim \Bin(N,U)$, 2) $X \mid U \sim \Bin(n,U)$, 3)~ $p - X \mid U \sim \Bin(N-n,U)$, and 4) $X \indep p-X \mid U$. 

The only nontrivial claim is (4). This property is proved in \cite[Table 2]{neufeld2024data} and is the hypergeometric analog of Poisson thinning. Expanding \eqref{eq:Bayes_risk} by the law of total variance,
\begin{align*}
    \E[\Var( p \mid X)] & = \E[\Var(\E[p \mid U,X] \mid X)] + \E[\Var(p\mid U,X)] \\
    & = \E[\Var(\E[p - X \mid U,X] \mid X)] + \E[\Var(p - X\mid U,X)] \\
    & = \E[\Var(\E[p - X \mid U] \mid X)] + \E[\Var(p - X\mid U)] \\
    & = (N-n)^2\E[\Var(U \mid X)] + (N-n)\E[U(1-U)].
\end{align*}
Because $X \mid U \sim \Bin(n,U)$ and $U \sim \Beta(\sqrt{n}/2,\sqrt{n}/2)$, Beta-binomial conjugacy shows $U \mid X \sim \Beta(\sqrt{n}/2 + X, \sqrt{n}/2 + n - X)$.  By the useful relation in Lemma \ref{lemma:beta_identity} between the mean and variance, we obtain
\begin{align*}
&\E[\Var( p \mid X)]  = (N-n)^2\frac{\E\left[\E[U(1-U) \mid X]\right]}{\sqrt{n} + n} + (N-n)\E[U(1-U)] \\
& = \left(\frac{(N-n)^2}{\sqrt{n} + n} + (N-n) \right)\E[U(1-U)] 
 =  \left(\frac{(N-n) + (n + \sqrt{n})}{\sqrt{n} + n} \right)\frac{(N-n)\sqrt{n}}{4(\sqrt{n} + 1)} 
 = \frac{(N-n)(N +\sqrt{n})}{4(\sqrt{n} + 1)^2}.
\end{align*}
This proves the lower bound in Eq. \eqref{eq:survey_sampling_risk_bounds}.

\end{proof}

\section{Miscellaneous Lemmas}

\begin{lemma}[Stein's Lemma for Rademacher Sums]
\label{lemma:stein_rademacher_sums}
Let $\e_1,\dots,\e_n$ be i.i.d. Rademacher random variables, and for any $1 \leq k \leq n$, let $U_k = \sum_{i=1}^k \e_i$. For any function $f,$ it holds that
\[
\E[U_nf(U_n)] = \frac{n}{2}\left(\E[f(U_{n-1} + 1)] - \E[f(U_{n-1} - 1)]\right).
\]
\end{lemma}
\begin{proof}
This is a direct consequence of exchangeability.
\end{proof}

    

\begin{lemma}[Finite Minimax Theorem]
\label{lemma:existence_of_lfp}
Let $\cl{X}$ and $\Theta$ be finite spaces.
For each $\theta \in \Theta$, let $P_\theta$ be a probability distribution on $\cl{X}$ and $X \sim P_\theta$. Let $g(\theta) \in [-1,1]$ be an estimand. For any estimator $\delta:\cl{X} \rightarrow [-1,1]$ define $R(\delta,g(\theta)) = \E_\theta(\delta(X) - g(\theta))^2.$ Then
\begin{equation}
\inf_\delta \sup_{\theta \in \Theta} R(\delta,g(\theta)) = \sup_{\pi \in \Delta(\Theta)} \inf_\delta \E_{\theta \sim  \pi} R(\delta,g(\theta)).
\end{equation}
There exists a least favorable prior for which the right-hand side supremum is attained, as well as a minimax estimator $\delta$ such that the left-hand side infimum is attained.
\end{lemma}
\begin{proof}
The minimax theorem essentially follows from results of Wald \cite{wald1950statistical}, but we provide a self-contained proof. Define $L(\pi,\delta) = \E_{\theta \sim \pi} R(\delta,g(\theta))$. We apply Sion's minimax theorem, noting that $\Delta(\Theta)$ is compact convex and the space of estimators $[-1,1]^{\cl{X}}$ is compact convex. Moreover, $L$ is continuous, linear in $\pi$, and convex in $\delta$. Then
\begin{align*}
\sup_{\pi} \inf_{\delta} L(\pi,\delta) & = \inf_\delta \sup_{\pi} L(\pi,\delta)  = \inf_\delta \sup_\theta R(\delta,g(\theta)).
\end{align*}
The supremum over $\pi$ is attained by compactness and the fact that $\pi \mapsto \inf_\delta L(\pi,\delta)$ is continuous by Berge's maximum theorem \cite{berge1963topological}.
The infimum over $\delta$ in $\inf_\delta \sup_{\theta \in \Theta} R(\delta,g(\theta))$ is attained for similar reasons. In particular, $\delta \mapsto \sup_{\theta} R(\delta,g(\theta))$ is continuous as a maximum of finitely many continuous functions, and the domain is compact.
\end{proof}

\begin{lemma}
\label{lemma:beta_identity}
Let $U \sim \Beta(\alpha,\beta)$. Then $\Var(U) = \E[U(1-U)]/(\alpha + \beta).$
\end{lemma}
\begin{proof}
Standard references show that $\E[U] = \alpha/(\alpha + \beta)$ and $\Var[U] = \frac{\alpha \beta}{(\alpha + \beta)^2(\alpha + \beta + 1)}$. We compute
\begin{align*}
    \E[U - U^2] & = \E[U] - \E[U]^2 - \Var[U]  = \frac{\alpha^2 + \alpha \beta}{(\alpha + \beta)^2} - \frac{\alpha^2}{(\alpha + \beta)^2} - \frac{\alpha \beta}{(\alpha + \beta)^2(\alpha + \beta + 1)} \\
    & = \frac{\alpha \beta}{(\alpha + \beta)^2} \left( 1 - \frac{1}{\alpha + \beta + 1} \right) = \frac{\alpha \beta}{(\alpha + \beta)(\alpha + \beta + 1)},
\end{align*}
as desired.
\end{proof}

\begin{lemma}
\label{lemma:variance_calculation}
In the setting of the proof of Theorem \ref{thm:survey_experiment_extension}, let $Z = (2X - n)/n$. Then
\[
\Var(Z) = \frac{1}{n}\left(1 - \frac{(n-1)}{N-1}\frac{p+q}{N} - \frac{N-n}{N-1}\tau^2\right).
\]
\end{lemma}
\begin{proof}
Recall $\tau = (p-q)/N$. Using the law of total variance,
\begin{align*}
    \Var(Z) & = \E\Var(Z\mid K) + \Var(\E[Z \mid K]) = \frac{4(N-n)}{n(N-1)}\frac{\E[K(N-K)]}{N^2}  + \frac{4}{N^2}\Var(K) \\
               & = \frac{4(N-n)}{n(N-1)}\frac{N(p+r/2) - (p + r/2)^2 - r/4}{N^2}  + \frac{r}{N^2} \\
               & = \frac{N-n}{n(N-1)}\cdot \frac{2N^2(\tau + 1) - N^2(\tau + 1)^2 - r}{N^2} + \frac{r}{N^2} \\
               & = \frac{N-n}{n(N-1)}\cdot\left( 1 - \tau^2 - \frac{r}{N^2} \right)+ \frac{r}{N^2} = \frac{N-n}{n(N-1)}\cdot\left( 1 - \tau^2 \right)+ \frac{(n-1)}{nN(N-1)}r \\
               & = \frac{1}{n}\left(1 - \frac{(n-1)}{N-1}\frac{p+q}{N} - \frac{N-n}{N-1}\tau^2\right).
\end{align*}
\end{proof}

\begin{lemma}
\label{lemma:binomial_inverse_moment}
Let $K \sim \Bin(n,1/2)$. Then
\[
a_n := \E[K^{-1}; K \notin \set{0,n}] = \frac{2}{n} + \frac{2}{n^2} + O(n^{-3}).
\]
\end{lemma}
\begin{proof}
Let $S_n := \sum_{j=1}^n \frac{1}{j} \binom{n}{j}.$ By Pascal's identity, 
\begin{align*}
S_n - S_{n-1} & = \sum_{j=1}^{n-1} \frac{1}{j} \binom{n-1}{j-1} + \frac{1}{n} 
 = \sum_{j=1}^{n} \frac{1}{j} \binom{n-1}{j-1} 
 = \frac{1}{n} \sum_{j=1}^{n} \binom{n}{j} 
 = \frac{2^n - 1}{n}.
\end{align*}
Telescoping then yields $S_n = \sum_{k=1}^n \frac{2^k-1}{k}$. Therefore, with $H_n$ being the harmonic sum,
\begin{align*}
a_n & = \frac{1}{2^n} \left(\sum_{k=1}^n \frac{2^k-1}{k} - \frac{1}{n}\right) 
 = \sum_{l = 0}^{n-1} \frac{2^{-l}}{n - l} - 2^{-n} \left(H_n + \frac{1}{n}\right).
\end{align*}
The contribution from the $l \geq n/2$ tail in the sum above is exponentially small, so we can absorb it and the term $2^{-n} \left(H_n + \frac{1}{n}\right)$ into an $O(n^{-3})$ remainder.  From the expansion $\frac{1}{n-l} = \frac{1}{n}\cdot \frac{1}{1 - l/n} = \frac{1}{n} + \frac{l}{n^2} + O\left( \frac{l^2}{n^3}\right)$, we have
\begin{align*}
a_n & = \frac{1}{n}\sum_{l=0}^{n/2} 2^{-l} + \frac{1}{n^2}\sum_{l=0}^{n/2} l2^{-l} + O(n^{-3}) 
 = \frac{1}{n}\sum_{l=0}^{\infty} 2^{-l} + \frac{1}{n^2}\sum_{l=0}^{\infty} l2^{-l} + O(n^{-3}) 
 = \frac{2}{n} + \frac{2}{n^2} + O(n^{-3}).
\end{align*}

\end{proof}

\end{document}